\documentclass[a4paper,11pt]{amsart}

\usepackage{amscd}
\usepackage{xypic}  
\usepackage{amssymb}
\usepackage{amsthm}
\usepackage{epsfig}
\usepackage{comment}
\usepackage[T1]{fontenc}

\newtheorem{thm}{Theorem}[section]
\newtheorem{cor}[thm]{Corollary}
\newtheorem{lemma}[thm]{Lemma}
\newtheorem{prop}[thm]{Proposition}

\renewcommand{\proofname}{Proof}
\newtheorem{proposition}[thm]{Proposition}

\newtheorem{question}[thm]{Question}

\theoremstyle{definition}
\newtheorem{remark}[thm]{Remark}
\newtheorem{definition}[thm]{Definition}

\def\min{\operatorname{min}}

\def\max{\operatorname{max}}

\def\c1{\operatorname{c_1}}
\def\c2{\operatorname{c_2}}

\def\Hilb{\operatorname{Hilb}}
\def\Sym{\operatorname{Sym}}

\def\CC{{\mathbb C}}
\def\ZZ{{\mathbb Z}}

\def\PP{{\mathbb P}}
\def\A{{\mathcal A}}

\def\SS{{\mathcal S}}
\def\G{{\mathcal G}}

\def\L{{\mathcal L}}
\def\M{{\mathcal M}}

\def\O{{\mathcal O}}

\def\E{{\mathcal E}}

\def\H{{\mathcal H}}
\def\F{{\mathcal F}}
\def\K{{\mathcal K}}

\def\V{{\mathcal V}}

\def\C{{\mathcal C}} 
\def\K{{\mathcal K}}
\def\X{{\mathcal X}}
\def\FF{{\mathbb F}}
\def\x{\times}                   
\def\cong{\simeq}

\def\sub{\subseteq}

\def\+{\oplus}                   
\def\*{\otimes}                  
\def\khpil{\rightarrow}

\def\Pic{\operatorname{Pic}}
\def\Supp{\operatorname{Supp}}

\def\Supp{\operatorname{Supp}}

\begin{document}

\title[ $k$-gonal loci in Severi varieties on $K3$ surfaces]{On $k$-gonal loci in Severi varieties on  general $K3$ surfaces and rational curves on hyperk\"ahler manifolds} 

\author{Ciro Ciliberto}
\address{Ciro Ciliberto, Dipartimento di Matematica, Universit\`a di Roma Tor Vergata, Via della Ricerca Scientifica, 00173 Roma, Italy}
\email{cilibert@axp.mat.uniroma2.it}

\author{Andreas Leopold Knutsen}
\address{Andreas Leopold Knutsen, Department of Mathematics
University of Bergen,
Postboks 7800, N-5020 Bergen, Norway}
\email{andreas.knutsen@math.uib.no}

\begin{abstract} In this paper we study the gonality of the normalizations of curves in the linear system $\vert H\vert$ of a general primitively  polarized complex $K3$ surface $(S,H)$ of genus $p$.  We prove two main results. First we give a necessary  condition on  $p, g, r, d$ for the existence of a curve in $ \vert H\vert$ with geometric genus $g$ whose normalization has a $g^ r_d$. Secondly we prove that for all numerical cases compatible with the above necessary condition, there is a family of \emph{nodal} curves in $\vert H\vert$ of genus $g$ carrying a $g^1_k$ and of dimension equal to the \emph{expected dimension} $\min\{2(k-1),g\}$. Relations with the Mori cone of the hyperk\"ahler manifold $\Hilb^ k(S)$  are discussed.
\end{abstract}

\maketitle

\tableofcontents

\section*{Introduction}
Let  $(S,H)$  be a primitively  polarized complex $K3$ surface of genus $p\ge 2$, i.e. $\Omega^ 2_S\cong \O_S$, $h^ 1(\O_S)=0$ and  $H$ is a globally generated, indivisible, divisor (or line bundle) with $H^ 2=2p-2$.  The main objective of this paper is to study the gonality of the normalization of curves, specifically of \emph{nodal} curves,  in the linear system $\vert H\vert$, when $(S,H)$ is general 
in its moduli space, that is, $(S,H)$ belongs to a Zariski open dense subset.


Let $V_{|H|,\delta}(S)\subseteq \vert H\vert$ be the \emph{Severi variety} of curves with $\delta\le p$ nodes. It is a classical result that $V_{|H|,\delta}(S)$ is a nonempty, locally closed, smooth variety of dimension $g=p-\delta$, which is the geometric genus of the curves in  $V_{|H|,\delta}(S)$.  The moduli morphism $V_{|H|,\delta}(S)\to \mathcal M_g$ is finite to its image (see Proposition \ref {prop:fff} below). 

We consider $V^k_{|H|,\delta}(S)\subseteq V_{|H|,\delta}(S)$ the subvariety of curves  whose normalizations carry a $g^1_k$. By Brill-Noether theory, if $g \leq 2(k-1)$ then  $V^k_{|H|,\delta}(S)= V_{|H|,\delta}(S)$  so the interesting range is $g >2(k-1)$. A count of parameters, carried out in \S \ref {ssec:loci}, suggests that the \emph{expected dimension} of  $V^k_{|H|,\delta}(S)$ is $2(k-1)$.  In any event, if nonempty, $V^k_{|H|,\delta}(S)$ has dimension at least $2(k-1)$ 
(see Proposition \ref {degarg}). 

Our first  main result is Theorem \ref{thm:noexist}, which yields a necessary condition for the normalization of a curve $C\in \vert H\vert$ of geometric genus $g$ (with any type of singularities) on a primitively polarized $K3$ surface $(S,H)$ of genus $p$ with no reducible curves in $|C|$ (e.g., with
$\Pic (S) \cong \ZZ[H]$) to possess a $g^r_d$, namely that 
\[ \rho(p,\alpha r,\alpha d+\delta)  \geq 0, \; \; \mbox{where} \; \; \alpha:=\Big\lfloor \frac{gr+(d-r)(r-1)}{2r(d-r)}\Big\rfloor, \]
and $\rho$ is the usual {\it Brill-Noether number}. In the case of $g^1_k$'s, setting $r=1$, $d=k$ and $\delta:=p-g$,
the above necessary condition reads
  \begin{equation}
    \label{eq:boundintro}
    \delta \geq \alpha \Big(p-\delta -(k-1)(\alpha+1)\Big)\;\; \text{where}\; \; \alpha:=\Big\lfloor \frac{p-\delta}{2(k-1)}\Big\rfloor.
  \end{equation}
Theorem \ref{thm:noexist} is a strong improvement of \cite[Thm.~1.4]{fkp} and its proof is based on the vector bundle  approach \`a la Lazarsfeld \cite{L}.

Our second main result deals with nonemptiness and dimension  of 
$V^k_{|H|,\delta}(S)$, and with properties of its general member, and proves that the bound \eqref{eq:boundintro} is optimal:

\begin{thm} \label{thm:main}
  Let $(S,H)$ be a general primitively polarized $K3$ surface of genus $p \geq 3$
  and let $\delta$ and $k$ be integers satisfying $0 \leq \delta \leq p$ and $k \geq 2$.  Set $g=p-\delta$. Then:
  \begin{itemize}
\item [(i)]    $V^k_{|H|,\delta}(S) \neq \emptyset$ if and only if \eqref{eq:boundintro} holds;
\item [(ii)]  when nonempty, $V^k_{|H|,\delta}(S)$ has an irreducible component of  the \emph{expected dimension} $\min\{2(k-1),g\}$ whose general element  is an irreducible curve $C$ with normalization $\widetilde{C}$ of genus $g$ such that $\dim (W^1_k(\widetilde{C}))=\max\{0,\rho(g,1,k)=2(k-1)-g\}$;
\item [(iii)]  
in addition, when $g \geq 2(k-1)$ (resp. $g < 2(k-1)$), any (resp. the general) $g^1_k$ on $\widetilde{C}$ has simple ramification and all nodes of $C$ are non-neutral with respect to it. 
\end{itemize}
\end{thm}

As for statement (i), we note that nonemptiness when $p$ is even,
$\delta \leq \frac{p}{4}$ and $k \geq \frac{p}{2} +1-\delta$ has been
proved by Voisin \cite[pf. of Cor.~1, p.~366]{vo} by a totally
different approach.

Theorem \ref{thm:main} yields that, for fixed $\delta>0$, the {\it general} curve in (some component of) the Severi variety $V_{|H|,\delta}(S)$ has the gonality of a general genus $g$ curve, i.e.  $\lfloor(g+3)/2\rfloor$, but, for all $k$ satisfying \eqref{eq:boundintro},  there are substrata of dimension $2(k-1)$ of curves of lower gonality $k$. This is in contrast to the case $\delta=0$, where the gonality is constant and equal to $\lfloor(p+3)/2\rfloor$. 
Parts (ii) and (iii) of  Theorem \ref {thm:main} also yield that $\widetilde{C}$ has gonality $\min\{k,\lfloor(g+3)/2\rfloor\}$
 and enjoys properties of a general curve of such a gonality.

The proof of Theorem \ref {thm:main} relies on a rather delicate degeneration argument. Indeed, the general $(S,H)$ can be specialized to the case where $S$ contains a smooth rational curve $\Gamma$ of degree $p$ and such that $S$ can be embedded into $\PP^{2p-1}$ by the linear system $|H+\Gamma|$. This $S$ can be in turn  degenerated to a union $R$ of two smooth rational normal surfaces $R_1\cong \PP^1 \x \PP^1$ and $R_2 \cong \FF_2$
intersecting transversally along a smooth elliptic curve of degree $2p$ 
and such that the rational curve $\Gamma$ specializes to the negative section 
$\mathfrak{s}_2$ 
of $R_2$ and $H$ specializes to the line bundle $\O_R(1)\otimes \O_R(-\mathfrak{s}_2)$. This is proved in \S\ref{ssec:unionscr}.  In \S\S \ref  {sec:chains} and \ref {sec:limnod} we describe nodal curves on the limit surface $R$ that fill up  limit components of $V_{|H|,\delta}(S)$. 
In \S \ref {S:kn} we describe possible limits  of $V^ k_{|H|,\delta}(S)$; this latter analysis is one of the crucial points in the paper, and, to the best of our knowledge, is a nontrivial novelty.  If nonempty, these limit varieties have the {expected dimension}, and this yields nonemptiness and expected dimension for the general $S$ containing a smooth rational curve of degree $p$ as above, and therefore also for the general $S$ as in the statement of the theorem (see Proposition \ref {degargrel}). In \S
\ref{sec:finalexistence} we show  nonemptiness of the limit varieties and the required properties for the $g^1_k$ in the range \eqref {eq:boundintro}. 

We note that our two--step degeneration seems to be new and has the  property of independent interest that the stable model of the general hyperplane section of the limit surface is an irreducible rational nodal curve. We believe that this technique can be useful in other contexts. Specifically, for $K3$ surfaces this can be used to study Severi varieties of nodal curves, also in $|nH|$ for $n>1$.

Besides its intrinsic interest for Brill-Noether theory and moduli problems, the subject of this paper is  related to Mori theory and rational curves on the $2k$-dimensional hyperk\"ahler manifold $\Hilb^ k(S)$ parametrizing $0$-dimensional length $k$-subschemes of the $K3$ surface $S$.
A curve on  $S$ with a $g^1_k$ on its normalization determines a rational curve on $\Hilb^ k(S)$.
For the importance of rational curves on hyperk\"ahler  manifolds see, e.g., \cite{H1,H2,Bou,HT,HT2,HTint,W1,W2,WW} and \S\;\ref{S:ratcur2}.  In particular, rational curves determine the nef and ample cones.

The curves on $S$ in Theorem \ref{thm:main} determine a family of rational curves in $\Hilb^ k(S)$ of dimension $2(k-1)$, which is the expected dimension of any family of rational curves on a $2k$-dimensional  hyperk{\"a}hler manifold. In  \S \ref {S:ratcur} we determine their classes in $N_1(\Hilb^ k(S))$ (see Lemma \ref{classR}); in this computation the properties of the $g_k^ 1$ stated in part (iii) of Theorem \ref {thm:main}  play an essential role. 
The lower $\delta$ is, the closer the class is to 
the boundary of the Mori cone. As a consequence, we obtain necessary conditions for a divisor in $\Hilb^ k(S)$ to be nef or ample (see Proposition \ref{prop:conoampio}). For infinitely many $p,k$ we prove that the classes of the rational curves in $\Hilb^ k(S)$ we obtain from Theorem \ref{thm:main}  with $\delta$ minimal satisfying \eqref{eq:boundintro} (which we call {\it optimal classes})  generate extremal rays of the Mori cone of $\Hilb^k(S)$ (see Corollary \ref{cor:intminima} and Proposition \ref{prop:intzero2}). After the appearance of the first version of this paper on the web, this has been verified also in  the cases 
$p \leq 2(k-1)$, where
$\delta=0$, in \cite{BM} (see Proposition \ref{prop:BM}). To determine the Mori cone of $\Hilb^ k(S)$ for all $p,k$ one would have to extend our results to the nonprimitive cases $|nH|$, $n>1$. This is a difficult task, but should in principle be possible to treat with similar methods. We plan to do this in future research. 

In \S  \ref {S:ratcur2} we also relate our work to some interesting conjectures of Hassett and Tschinkel on the Mori cone of $\Hilb^k(S)$ (see in particular Remark \ref{rem:HT}) and of  Huybrechts and Sawon on Lagrangian fibrations (see in particular Corollary \ref{cor:neclagfib}). \medskip

Throughout this paper we work over $\mathbb C$. As usual, and as we did already in this Introduction, we may sometimes abuse notation and identify divisors with the corresponding line bundles, indifferently  using  the additive and the multiplicative notation. \medskip

\subsection*{Acknowledgements} This paper is the result of a long term collaboration between the authors. A substantial part of it was accomplished while they were both visiting the Mittag-Leffler Institute in Stockholm during the programme {\it Algebraic Geometry with a view towards applications} 
in the spring of 2011. Both authors are very grateful to this institution for the warm and fruitful hospitality. 

The authors  thank K.~Ranestad and F.~Flamini for useful conversations, L.~Benzo for comments on the first version of this paper, B.~Hassett for useful correspondence, A.~Bayer and E.~Macr{\`i} for informing us about and sending us their recent preprints \cite{BM,BM-MMP} and for correspondence about them,
and the referee for  the careful reading of the paper and  useful suggestions.

The first author is a member of the G.N.S.A.G.A. of the Istituto Nazionale di Alta Matematica ``F. Severi''.

\section{Severi varieties, $K3$ surfaces and $k$-gonal loci} \label{S:Sevvar}


\subsection{Severi varieties and $k$-gonal loci} Let $S$ be a connected, projective surface with normal crossing singularities and let $|H|$ be a base point free,
complete linear system of Cartier divisors on $S$ whose general element is a connected curve $H$ with at most
nodes as singularities, located at the singular points of $S$. We will set $p=p_a(H)$.

For any integer $0 \leq \delta \leq p$, we denote by
$V_{|H|, \delta}(S)$  the locally closed 
subscheme of $|H|$ parametrizing the universal family of
curves $C\in \vert H\vert$ having only nodes as singularities, exactly $\delta$ of them (called the \emph{marked nodes}) off
the singular locus of $S$, and such that the partial normalization $\widetilde C$ at these 
$\delta$ nodes is connected (i.e., the marked nodes are \emph{not disconnecting nodes}). We set $g=p-\delta=p_a(\widetilde C)$. 
If $S$ is smooth the $V_{|H|, \delta}(S)$'s  are called {\em Severi
varieties} of $\delta$-nodal curves in $|H|$ on
$S$. We use the same terminology in our more general setting.

Let $g \geq 3$ be an integer.  We denote by $\M_g$ the \emph{moduli
space (or stack)} of smooth curves of genus $g$, whose dimension is $3g-3$.
We recall that $\M_g$ is
quasi-projective and admits a projective compactification $\overline{\M_g}$, parametrizing all connected stable curves of arithmetic genus $g$.  

One has the \emph{moduli morphism}
\begin{equation}\label{eq:uno}
\xymatrix{\psi_{S, H,\delta} :V_{|H|, \delta}(S) \ar[r] &
\overline{\M_g}}
\end{equation}
sending $C\in V_{|H|,\delta}(S)$ to the isomorphism class of the stable model $\overline C$
of the partial normalization $\widetilde C$ of $C$ at the $\delta$ marked nodes.  We write $\psi$
rather than $\psi_{S,H,\delta}$ if no confusion arises.
If $\psi$ is generically finite to its image, 
we say that $V_{|H|,\delta}(S)$ has \emph{maximal number of moduli} $g$. 

One can consider the stratification of $\M_g$ in terms of
gonality
\[
\M^1_{g,2} \subset \M^1_{g,3} \subset \ldots \subset \M^1_{g,k}
\subset \ldots \subset \M_g,
\]
where
\[
\M^1_{g,k}:= \Big \{[Y] \in \M_g
| \; Y \; \mbox{possesses a} \;
g^1_k \Big \},
\]
called the $k$-{\em gonal locus} in $\M_g$,  is irreducible, of
dimension $2g + 2 k -5$ when $g\ge 2(k-1)$, whereas
$\M^1_{g,k}= \M_g$ when $g\le 2(k-1)$
(see e.g. \cite {AC}). Recall that $\psi(C)\in \overline{\M^1_{g,k}}$ if and only if the partial normalization $\widetilde C$ of $C$ 
at the $\delta$ marked nodes is stably equivalent to a curve that is the domain of an admissible cover of degree $k$ to a stable pointed curve of genus $0$ (see  \cite[Theorem (3.160)]{HM}).
For any 
integer $k \geq 2$, we define 
\begin{equation*} \label{eq:defvk}
V_{|H|, \delta}^k(S):=  \Big \{C \in V_{|H|, \delta}(S) \; | \;
\psi(C)\in \overline{\M^1_{g,k}}\Big \},
\end{equation*}
which has a natural scheme structure.
This is called the \emph{$k$-gonal locus} inside $V_{|H|, \delta}(S) $.

\subsection{$K3$ surfaces} \label{sub:k3} We will mainly consider the case in which $S$ 
is a smooth, projective $K3$ surface, endowed with a globally generated \emph{primitive}, i.e. indivisible,
divisor $H$ with  $p=p_a(H) \geq 2$.  We call $(S,H)$
a \emph{primitive (or primitively polarized)  $K3$ surface of genus} $p$. We denote by 
$\K_p$ the \emph{moduli space (or stack)} of primitive $K3$ surfaces of genus $p$,
which is smooth and irreducible of dimension $19$, and the general element 
$(S,H)$ is such that $H$ is very ample. Furthermore, 
${\rm Pic}(S)$ is generated by the class of $H$ for all $(S,H)$ outside a countable union of Zariski closed proper subsets (the Noether-Lefschetz divisors).

If $V_{|H|, \delta}(S)  \neq \emptyset$, then it is {\em
regular}, i.e.~it is smooth and of the \emph{expected dimension} $g$. 
Indeed, the marked, not disconnecting  nodes of the curves in $V_{|H|, \delta}(S)$ impose
independent conditions to the linear system $|H|$ (see e.g. \cite{CS}).
If $V_{|H|, \delta}(S) \neq \emptyset$ and $\delta' < \delta$, then 
$V_{|H|,\delta}(S)\subset \overline {V_{|H|, \delta'}}$.

\begin{remark}\label{rem:sing} The latter holds for $V_{|H|, \delta}(S)$ also when $S$ is a connected surface with  local normal crossing
singularities, trivial dualizing bundle, $h^1(S,\mathcal O_S)=0$ and $H$
a globally generated, primitive divisor on $S$.  Indeed, the usual  arguments (like  in  \cite{CS}) apply with no change.  
 \end{remark}

By a result of
Mumford's (cf. \cite[Appendix]{MM}), for all $\delta\le p$ the 
Severi varieties $V_{|H|, \delta}(S) $ are nonempty. Chen extended this  to Severi varieties $V_{|mH|, \delta}(S)$
with $m>1$ (cf. \cite{C}). 

The following proposition is related to the rigidity results in \cite{halic}.

\begin{prop}\label{prop:fff} Let $(S,H)\in \K_p$. The differential of the moduli morphism $\psi_{S,H,\delta}$ is everywhere injective, hence all components of $V_{\vert H\vert,\delta}(S)$ have maximal number of moduli.
\end{prop}

\begin{proof} Let $C$ be a curve in $V_{\vert H\vert,\delta}(S)$
and let $f:\widetilde C\to C$ be the normalization at the $\delta$ nodes. We have the following exact sequence
\[
\xymatrix{
0 \ar[r] &  T_{\widetilde C} \ar[r] &  f^ *(T_S) \ar[r] &  N_f \ar[r] &  0,
}
\]
which defines the normal sheaf $N_f$ to the map $f$. The differential of $\psi$ at ${\widetilde C}$ is the coboundary map  $H^ 0({\widetilde C},N_f)\to H^ 1({\widetilde C},T_{\widetilde C})$. Hence it suffices to prove that
$h^ 0({\widetilde C},f^ *(T_S))=0$, i.e.~that $h^ 0({\widetilde C},\varphi^ *(T_S)_{\vert \widetilde C})=0$, where $\varphi: \widetilde S\to S$ is the blow-up of $S$ at the nodes of $C$. Denote by $E_i$ the exceptional divisors, with $1\le i\le \delta$. Consider the diagram with exact rows and  columns
\begin{equation}\label {eq:3}
\xymatrix{
    &    0 \ar[d] & 0 \ar[d] & 0 \ar[d] &  \\
0  \ar[r] &   T_{\widetilde S}(-\widetilde C) \ar[d] \ar[r] & \varphi^*(T_S)(-\widetilde C) \ar[d] \ar[r] &\oplus_{i=1}^ \delta  \O_{\PP^ 1}(-1) \ar[d]  \ar[r] & 0 \\
0  \ar[r] &   T_{\widetilde S}\ar[d] \ar[r] & \varphi^*(T_S)\ar[d] \ar[r] &  \oplus_{i=1}^ \delta   \O_{\PP^ 1}(1)\ar[d]  \ar[r] & 0 \\
0\ar[r]& {T_{\widetilde S}}_{|\widetilde C}\ar[d] \ar[r] & \varphi^*(T_S)_{\vert \widetilde C}\ar[d] \ar[r] & \Delta\ar[d]  \ar[r] &0 \\
&0 & 0& 0& }
\end{equation}
where $\Delta\cong \CC^ {2\delta}$ is a skyscraper sheaf of rank 1 supported at the $2\delta$ intersections  of $\widetilde C$ with the $E_i$'s and the rightmost sheaves in the first two rows are supported on $E_i\cong \PP^ 1$, for $1\le i\le \delta$. One has $h^0(\widetilde S, \varphi^*(T_S))=h^0(S, T_S)=0$ and $H^1(\widetilde S,  \varphi^*(T_S)(-\widetilde C) )\cong H^1(\widetilde S,T_{\widetilde S}(-\widetilde C))$. Grant for the time that
\begin{equation}
  \label{eq:inj}
  H^ 0(\widetilde C, {T_{\widetilde S}}_{|\widetilde C})=0.
\end{equation}
By \eqref {eq:inj}, 
the map $ H^ 1(\widetilde S,T_{\widetilde S}(-\widetilde C))\to H^ 1(\widetilde S, T_{\widetilde S})$ is injective. Its image $\mathfrak A$ corresponds to first order  deformations of $\widetilde S$ that keep $\widetilde C$ fixed. These deformations do not move the $E_i$'s, and therefore $\mathfrak A$ intersects the image of 
$H^ 0(\widetilde S, \oplus_{i=1}^ \delta   \O_{\PP^ 1}(1) )\to H^1(\widetilde S,  T_{\widetilde S})$ in $(0)$. This implies that the map $H^ 1(\widetilde S, \varphi^*(T_S)(-\widetilde C)) \to H^1(\widetilde S,  \varphi^*(T_S))$ is injective and  $h^0(\widetilde C, \varphi^*(T_S)_{\vert \widetilde C})=0$ follows. 

We now prove \eqref{eq:inj}. A local computation shows that $\varphi_*({T_{\widetilde S}}_{|\widetilde C})={T_S}_{|C}$, so it suffices to prove that
\begin{equation}\label{eq:1}
h^ 0(S, {T_S}_{|C})=0.
\end{equation}
Consider $S$ embedded in $\PP^p$ by $|H|$. By the exact sequence
\[
\xymatrix{
0\ar[r] & {T_S}_{|C}\ar[r] & {T_{\PP^ p}}_{|C}\ar[r] & {N_{S/\PP^ p}}_{|C}\ar[r] & 0,
}\]
to prove \eqref{eq:1} one has to prove that the map
$\gamma: H^ 0(C, {T_{\PP^ p}}_{|C}) \to H^0(C, {N_{S/\PP^ p}}_{|C})$
is injective. 

The cohomology of the Euler sequence
\[ \xymatrix{
0\ar[r] & \O_S\ar[r] & H^ 0(S,\O_S(C))^*\otimes \O_S(C)\ar[r] & {T_{\PP^ p}}_{|S} \ar[r] & 0
} \]
yields that
$ H^ 0(S,{T_{\PP^ p}}_{|S})\cong  H^ 0(\PP^ p,T_{\PP^ p})\cong \CC^ {p^ 2+2p}$
 is the tangent space to  ${\rm PGL}(p+1, \CC)$. Similarly
 $ H^ 0(S,{T_{\PP^ p}}_{|S}(-C) )\cong    \CC^ {p+1}$
 is the tangent space to the subgroup of ${\rm PGL}(p+1, \CC)$ that pointwise fixes the hyperplane in which $C$ lies.

Consider the long exact cohomology sequence associated to the Euler sequence for $C$, i.e.
\[
\xymatrix{
0\ar[r]&H^ 0(C, \O_C) \ar[r] & H^ 0(S,\O_S(C))^*\otimes H^ 0(C,\omega_C)\cong \CC^ {p(p+1)} \ar[r]& H^ 0(C,{T_{\PP^ p}}_{|C})  & \\
\ar[r] &  H^ 1(C, \O_C)\ar[r]& H^ 0(S,\O_S(C))^*\otimes H^ 1(C,\omega_C).
 }
\]
The map in the second row is dual to the surjective map
\[ \xymatrix{
H^ 0(S,\O_S(C))\otimes H^ 1(C,\omega_C)  \cong  H^ 0(C,\omega_C)\oplus \CC  \ar[r] &  H^ 0(C,\omega_C),
} \]
hence the last map in the first row  is surjective, so that $H^ 0(C,{T_{\PP^ p}}_{|C})\cong \CC^ {p^ 2+p-1}$.

Consider the commutative  diagram with exact rows and columns
\[
\xymatrix{
&& 0 \ar[d] & 0\ar[d]\\
&0\ar[r] & H^ 0(S, {T_{\PP^ p}}_{|S} (-C)) \cong \CC^ {p+1} \ar[d]  \ar[r]^{\hspace{0.6cm}\alpha}   & H^0(S, N_{S/\PP^ p} (-C)) \ar[d] \\
&0\ar[r] &H^ 0(S, {T_{\PP^ p}}_{|S})\cong \CC ^ {p^ 2+2p}   \ar[d] \ar[r]^{\hspace{0.6cm}\beta}  & H^0(S, N_{S/\PP^ p}) \ar[d] \\
&&H^ 0(C, {T_{\PP^ p}}_{|C}) \ar[d] \cong \CC^{p^ 2+p-1} \ar[r]^{\hspace{0.6cm}\gamma}  & H^0(C, {N_{S/\PP^ p}}_{|C} ) \\
&&0 &&
}
\]
Assume we have $x\in H^ 0(C, {T_{\PP^ p}}_{|C})$ such that $\gamma(x)=0$. Lift $x$ to $y\in H^ 0(S, {T_{\PP^ p}}_{|S})$. Then
$\beta(y)\in H^0(S, N_{S/\PP^ p} (-C))$. The above geometric interpretation  tells  us  that $y\in H^ 0(S, {T_{\PP^ p}}_{|S} (-C))$, proving that $x=0$, which implies the injectivity
of $\gamma$, hence \eqref{eq:1}. 
\end{proof}

Later we will need to consider the substack 
$\K'_p$ of $\K_p$ consisting of pairs $(S,H)$ such that $S$ contains a smooth rational curve $\Gamma$ satisfying $\Gamma\cdot H=p$. 

\begin{prop}\label{prop:ssub} For any $p\ge 2$,  $\K'_p$ is irreducible of codimension one in $\K_p$, and the general element $(S,H)\in \K'_p$ is such that $H$
is ample. Furthermore, $\Pic(S)\cong \ZZ[H]\oplus \ZZ[\Gamma]$ for all $(S,H)$ outside a countable union of Zariski closed proper subsets of $\K'_p$.
\end{prop}
\begin{proof} This is standard: the proof follows much the same arguments as, e.g., \cite [Proposition (3.2)]{clm1}, using \cite [Theorem 1.14.4]{nik} and \cite [pp. 271--2]{Do}.  
\end{proof}

We note that $\K'_p$ is a Noether-Lefschetz divisor in $\K_p$. 

\subsection{Universal Severi varieties and degenerations}\label{ssec:USV}

For any $p\ge 2$,
 and $0\le \delta \le p$, one can consider 
a stack $\V_{p,\delta}$ (see \cite [Proposition 4.8] {fkps}), called
the  \emph{universal Severi variety}, which is pure and smooth of
dimension $19+g$, 
endowed with a morphism $\phi_{p,\delta}: \V_{p,\delta}\to
\K_p$ and its fibres are so described
\[
\xymatrix@C=10pt{
\V_{p,\delta} \ar@{}[r]|(0.4){\supset} \ar[d]_{\phi_{p,\delta}}
& V_{\vert H\vert,\delta}(S) \ar[d] \\  
\K_p \ar@{}[r]|(0.4){\ni} & (S,H)
}
\]
Some fibers may be empty, but there is a dense open
substack $\K_p^ {\circ}$ of $\K_p$ over which the fibers are nonempty and the morphism $\phi_{p,\delta}: \V_{p,\delta}\to \K_p^ {\circ}$ is smooth on all components of
$\V_{p,\delta}$, each  dominating $\K_p^ {\circ}$. 

In a similar way one can consider the \emph{$k$-gonal universal locus}
$\V^ k_{p,\delta}\subseteq \V_{p,\delta}$. 

We will need this in a more general setting. Suppose
we have a proper flat family of surfaces 
$f :\SS\to \mathbb D$, where $\mathbb D$ is a disc.
Assume that:
\begin{itemize}
\item $\SS$ is smooth, endowed with a line bundle $\H$;
\item $f$ is smooth over $\mathbb D^ *=\mathbb D-\{0\}$;
\item if $t\in \mathbb D^ *$, then the fibre $S_t$ of $f$ over $t$ is a $K3$ surface;
\item the fibre $S_0$ of $f$ over $0$ is a local normal crossing divisor in $\SS$;
\item the line bundle $\H_t:=\H_{\vert S_t}$ determines a complete linear system
$\vert H_t\vert$ of dimension $p$ for all $t\in \mathbb D$ and $(S_t,H_t)\in \K_p$
for all $t\in \mathbb D^ *$.
\end{itemize}

Since $\V _{p,\delta}$ is functorially defined, we have
\emph{$f$-relative Severi varieties} $\phi_{f;p,\delta}: \V_{f;p,\delta}\to \mathbb D^ *$, 
with $\V_{f;p,\delta}$ locally closed in $\PP(f_*(\H))$, such that the fibre of $\phi_{f;p,\delta}$ over $t$ is $V_{\vert H_t\vert,\delta}(S_t)$ for all $t\in \mathbb D^ *$.   
We will drop the index $\delta$ when $\delta=0$. 

\begin{lemma}\label{lem:defo} Let $C_0\in \vert H_0\vert$
be an element of $V_{|H_0|, \delta}(S_0)$, with
 $\delta$ not disconnecting nodes $q_1,\ldots,q_\delta$ off the singular locus of $S_0$. 
Then $C_0$ sits in the closure of
$\V_{f;p,\delta}$ in $\PP(f_*(\H))$ and $\V_{f;p,\delta}$ dominates $\mathbb D$. 
\end{lemma}

\begin{proof} 
We have a commutative diagram with exact rows and  column
\[
\xymatrix{
&&& 0 \ar[d] &  \\
0\ar[r] &T_{[C_0]}(V_{|H_0|, \delta}(S_0)) \cong  H^0(C_0, N'_{C_0/S_0}) \ar[r] & \mathbb{C}^p \cong T_{[C_0]}(|H_0|)  \cong \hspace{-1.2cm} & H^0(C_0, N_{C_0/S_0}) \ar[d] \ar[r]^{\hspace{0.3cm}\alpha} &  \+_{i=1}^{\delta} T^1_{q_i} \ar@{=}[d] \\
 0\ar[r] & H^0(C_0, N'_{C_0/\mathcal S}) \ar[r] &\mathbb{C}^{p+1} \cong T_{[C_0]} (\mathcal{V}_{f;p})  \cong \hspace{-1.05cm} &H^0(C_0, N_{C_0/\mathcal S}) \ar[d] \ar[r]^{\hspace{0.3cm}\beta} &  \+_{i=1}^{\delta} T^1_{q_i} \\
&&& H^0(C_0, \O_{C_0})  \ar[d] & \\
&&& 0 &
}
\]
where $N'_{C_0/S_0}$ and $N'_{C_0/\mathcal S}$ are the equisingular normal sheaves at the marked nodes of $C$ in $S_0$ and $\mathcal S$, respectively. 
By hypothesis, $\alpha$ is onto, hence so is $\beta$. Thus $H^0(C_0, N'_{C_0/\mathcal S})$, which is the tangent space at $[C_0]$ of the space of equisingular deformations at the $\delta$ nodes of $C_0$ in $\mathcal S$, has dimension $\dim (V_{|H_0|, \delta}(S_0))+1$, and the assertion follows.
\end{proof}

\subsection{$K3$ surfaces and $k$-gonal loci}\label{ssec:loci}
Let  $(S,H)\in \K_p$ be general.  By Brill-Noether theory, 
$V^k_{|H|,\delta}(S) = V_{|H|,\delta}(S)$ if $\delta \geq p-2(k-1)$.

\begin{prop} \label{degarg}   Let $(S,H)$ be in $\K_p$. Assume $g:=p-\delta \ge 2(k-1)$. Then for any  irreducible component $V$ of $V_{|H|,
\delta}^k(S)$ one has $\dim (V) \geq 2(k-1)$.
\end{prop}

\begin{proof}  Consider the morphism $\psi$ in \eqref {eq:uno}.
Let $V$ be an irreducible component of $V_{|H|, \delta}^k(S)$ and 
$V'$ the $g$-dimensional,  irreducible component of 
$V_{|H|, \delta}(S)$ containing it, so that
\[ \emptyset \neq \psi(V) \sub \psi(V') \cap \M^1_{g,k}. \]
Set $W=\overline{\psi(V)}$ and $W'=\overline{\psi(V')}$, 
so that $W$ is an  irreducible component of $W' \cap \overline{\M^1_{g,k}}$ and,
by Proposition \ref {prop:fff}, one has $\dim(W')=g$. 
Then
\begin{eqnarray*}
\dim(V)  \geq  \dim (W) \geq  \dim(W')+ \dim (\overline {\M^1_{g,k}})- \dim (\overline{\M_{g}})=2(k-1).
\end{eqnarray*}\end{proof}

The proof of Proposition \ref {degarg} shows that the \emph{expected dimension}
of an irreducible component  of $V^ k_{\vert H\vert,\delta}(S)$ is $\min\{2(k-1),p-\delta\}$.

It is convenient to have a \emph{relative version} of Proposition \ref {degarg}.
Let $f: \SS\to \mathbb D$ be  as in \S \ref {ssec:USV}.  One can define the 
\emph{$f$-relative $k$-gonal locus} $\V^ k_{f; p,\delta}\subseteq \V_{f;p,\delta}$
over $\mathbb D^ *$.

\begin{prop}\label {degargrel}
Let $V_0$ be a  component of $V_{\vert H_0\vert,\delta}^ k(S_0)$.  If $\dim(V_0)=2(k-1)$, then $V_0$ is contained in an irreducible component $\V$ of $\V^ k_{f; p,\delta}$
dominating $\mathbb D^ *$, with $\dim(\V)=\dim(V_0)+1$.
\end{prop}

\begin{proof}
  Similar to the proof of Proposition \ref {degarg}, using Lemma \ref{lem:defo} (see also \cite [Prop.~5.11 and its proof] {fkp2}).
\end{proof}


\section{Rational curves in the Hilbert scheme of points of a $K3$ surface} \label{S:ratcur}


If $S$ is a $K3$ surface, then $\Hilb^k(S)$ is a 
{\it hyperk\"ahler
manifold}, also called an {\it irreducible symplectic manifold} 
(see e.g. ~\cite{B,H1}). The cohomology group
$H^2(\Hilb^k(S), \mathbb Z)$ is endowed with the
{\em Beauville-Bogomolov quadratic form}  $q$ and one has 
the orthogonal decomposition 
\begin{equation} \label{eq:decH2}
H^2(\Hilb^k(S), \mathbb Z) \cong H^2(S, \mathbb Z) \oplus_{\perp} \mathbb Z [\mathfrak{e}_k],
\end{equation}
where $\Delta_k:=2\mathfrak{e}_k$ is the class of the divisor (still denoted by $\Delta_k$) parametrizing nonreduced
$0$-dimensional subschemes  \cite{B}. Equivalently, 
$\Delta_k$ is  the exceptional divisor of the
\emph{Hilbert-Chow morphism} $\mu_k: \Hilb^k(S) \khpil \Sym^k(S)$. 
The embedding of $H^2(S, \mathbb Z)$ into $H^2(\Hilb^k(S), \mathbb Z)$ in 
\eqref{eq:decH2} is given by sending a class $F \in H^2(S, \mathbb Z)$
to the class in $H^2(\Hilb^k(S), \mathbb Z)$ determined by all subschemes whose support intersects a representative of $F$. By abuse of notation we  will still denote  by $F$ this class in $H^2(\Hilb^k(S), \mathbb Z)$. The restriction of the Beauville-Bogomolov form to $H^2(S, \mathbb Z)$ is  the cup product on
$S$, and $q(\mathfrak{e}_k)=-2(k-1)$. Accordingly,
\eqref{eq:decH2} induces an  orthogonal decomposition (see  \cite{B})
\begin{equation} \label{eq:decpic}
\Pic (\Hilb^k(S)) \cong \Pic (S)  \oplus_{\perp} \mathbb Z [\mathfrak{e}_k].
\end{equation}

Given a primitive class $\alpha\in H_2(\Hilb^k(S), \mathbb Z)$, there exists a unique class  $w_{\alpha}\in
H^2(\Hilb^k(S), \mathbb Q)$ such that $\alpha \cdot v=q(w_{\alpha},v)$, for all $v\in
H^2(\Hilb^k(S), \mathbb Z)$, and one sets (cf. e.g. \cite{HTint})
\begin{equation}\label{eq:qalfa}
q(\alpha):= q(w_{\alpha}).
\end{equation}
This gives a $\mathbb Q$-valued form on homology, and we have
\begin{equation}\label{eq:cacca}
H_2(\Hilb^k(S), \mathbb Z) \cong H_2(S, \mathbb Z) \oplus_{\perp} \mathbb Z [\mathfrak{r}_k],
\end{equation}
where $\mathfrak{r}_k$ is the homology class orthogonal to $H^2(S, \mathbb Z)$ and satisfying $\mathfrak{e}_k \cdot \mathfrak{r}_k=-1$, see e.g. \cite[\S 1]{HT}.
As explained in \cite[Ex. 4.2]{HTint}, $\mathfrak{r}_k$ is the class of a fiber of the Hilbert-Chow morphism, i.e.,  it is the class of the rational curve in $\Delta_k$ corresponding to the curve lying above 
$2x_1+x_2+ \cdots +x_{k-1}$ in $\Sym^k(S)$, for any $k-1$ distinct points 
$x_1, \ldots, x_{k-1}$ of $S$.  The embedding of $H_2(S, \mathbb Z)$ in $H_2(\Hilb^k(S), \mathbb Z)$ is given by sending the class of a cycle $Y$ to the class of the cycle 
\[
 \Big \{\xi\in \Hilb^k(S) | \Supp(\xi)=\{p_1, \ldots, p_{k-1}, y\}, \;  y\in Y \Big\},
\]
where $p_1, \ldots, p_{k-1}$ are distinct fixed points of $S$ off $Y$. 

The decomposition \eqref {eq:cacca} induces 
\[
N_1(\Hilb^k(S), \mathbb Z) \cong {\rm Pic}(S)\oplus_{\perp} \mathbb Z[\mathfrak{r}_k].
\]
If $R \equiv D-y\mathfrak{r}_k$ in  $N_1(\Hilb^k(S), \mathbb Z)$, with $D\in \Pic(S)$,   then 
\[
 w_R = D - \frac{y}{2(k-1)}\mathfrak{e}_k,
\]
and by \eqref{eq:qalfa}, one has
\begin{equation}\label{eq:qalfa2}
q(R)= 
D^2-\frac{y^2}{2(k-1)}.
\end{equation}

We mentioned in the Introduction the importance of rational curves on hyperk\"ahler  manifolds.
The relation with the topic of this paper is that a curve $C$ on a $K3$ surface whose normalization $\widetilde{C}$ possesses a $g^1_k$ gives rise, in an obvious way, to an irreducible rational curve $R$ in $\Hilb^k(S)$. 
Indeed, the $g^1_k=|A|$ on $\widetilde{C}$ induces a $\PP^1_{(C,A)} \subset \Sym^k(\widetilde{C})$
and this is mapped to an irreducible rational curve $\overline{R}_{(C,A)} \subset \Sym^k(S)$
by the composed morphism
\[
\xymatrix{
\Sym^k(\widetilde{C}) \ar[r] & \Sym^k(C)  \ar@{^{(}->}[r] & \Sym^k(S).
}
\]
The irreducible rational curve $R=R_{(C,A)} \subset \Hilb^k(S)$ is the strict transform
$(\mu_k)^{-1}_*( \overline{R}_{(C,A)})$ by the Hilbert-Chow morphism.

Let  $C$ be an element of ${V}^k_{|H|,\delta}(S)$, and assume that its normalization carries a $g^1_k$ such that
\begin{eqnarray}
  \label{eq:condnod1} &\mbox{ all the nodes of  $C$ are non-neutral with respect to the $g^1_k$;} & \\
 \label{eq:condnod2} &\mbox{ the $g^1_k$ has only simple ramification}. &
\end{eqnarray}
Let $R$ be the corresponding rational curve in $\Hilb^k(S)$.

\begin{lemma} \label{classR} Under hypotheses \eqref{eq:condnod1} and \eqref{eq:condnod2},
  the class\footnote{There is an erroneous fraction of $1/2$ in the corresponding formula for $k=2$ in \cite[(6.7)]{fkp2}, due to a trivial computational mistake in the line above the formula: $\PP^1_{\Delta}.\mathfrak{e}=-2$ should have been $-1$, where $\PP^1_{\Delta}$ is $\mathfrak{r}_k$ in our notation.} of $R$ in $N_1(\Hilb^k(S), \mathbb Z)$ is
$H- (g+k-1)  \mathfrak{r}_k$. 
\end{lemma}

\begin{proof} 
Write $R= H - y \mathfrak{r}_k$, so that $y=\mathfrak{e}_k  \cdot  R$.
Since all nodes are non-neutral and the $g^1_k$ has simple ramification everywhere, by Riemann-Hurwitz
we have
\[y= \mathfrak{e}_k  \cdot  R = \frac{1}{2}\Delta_k  \cdot  R= 
\frac{1}{2}(2g+2k-2)=g+k-1. \]
\end{proof}

The particular case $p=9$, $\delta=2$ and $k=4$ is treated in \cite[Ex. 4.5]{HTint}.

\begin{remark} \label{rem:classR} The conclusion of Lemma \ref{classR} holds even without hypothesis \eqref{eq:condnod2}. Furthermore, also hypothesis \eqref{eq:condnod1} can be weakened: if $\delta'$ is the number of non-neutral nodes of $C$, then the class of $R$ is $H- (g'+k-1)  \mathfrak{r}_k$, where $g':=p-\delta' \geq g$. For a detailed proof, we refer e.g. to 
\cite[\S\;3.3 and 5.3]{hui1} (see also \cite[\S\;6.2]{hui2}). 
\end{remark}


\section{Necessary conditions for existence of linear series on normalizations} \label{S:nonex}

Consider the usual  \emph{Brill-Noether number}
$\rho(g,r,d)=g-(r+1)(r+g-d)$.

\begin{thm} \label{thm:noexist}
  Let $(S,H)\in \K_p$ such that all elements in $|H|$ are reduced and irreducible (e.g., $\Pic (S) \cong \ZZ[H]$). Assume that $C \in |H|$ is a curve whose normalization possesses a $g^r_d$. Let $g$ be the geometric genus of $C$ and set $\delta=p-g$ and   $\alpha:=\Big\lfloor \frac{gr+(d-r)(r-1)}{2r(d-r)}\Big\rfloor$. Then
  \begin{equation}
    \label{eq:bound}
  \rho(p,\alpha r,\alpha d+\delta)  \geq 0,\; \text{i.e.,}\; \; \delta\ge \alpha \Big(rg-(d-r)(\alpha r+1)\Big).
      \end{equation}
\end{thm}

\begin{proof}
  Let $\nu: \widetilde{C} \to C$ be the normalization of $C$ and let $A$ be a line bundle on $\widetilde{C} $ such that $\vert A\vert = g^ r_d$. Then, for any positive integer $l$, the sheaf $\A_l:=\nu_*(lA)$ is torsion free of rank one on $C$ with $h^0(\A_l)=h^0(lA)\geq lr+1$ and $\deg \A_l=\deg (lA)+\delta=ld+\delta$ (see, e.g., \cite[Prop. 3.2]{fkp}). We have 
$\rho(\A_l) =\rho(p_a(C),h^0(\A_l)-1, \deg \A_l) \leq \rho(p,lr,ld+\delta)$ and we claim that
\begin{equation}
  \label{eq:rho}
  \rho(p,lr,ld+\delta)=l^ 2 r(d-r)-l(gr+r-d)+\delta  \geq 0.
\end{equation}

To prove \eqref{eq:rho}, let $\A_l'$ denote the globally generated part of $\A_l$, that is, the image of the evaluation map $H^0(\A_l) \*\O_C \to \A_l$. If \eqref{eq:rho} does not hold, then
$\rho(\A'_l) \leq  \rho(\A_l) <0$. In particular, $h^1(\A'_l) >0$.
The kernel of the evaluation map
 \[ 
\xymatrix{  H^0(\A_l) \* \O_S \ar[r] & \A_l' \ar[r] & 0}
\]
is a vector bundle, whose dual bundle $\E_l$ has rank $h^0(\A_l')=h^0(\A_l) \geq lr+1$ and satifies $c_1(\E_l)=C$ and 
$c_2(\E_l) = \deg \A_l' \leq \deg \A_l=ld+\delta$ (see  \cite{G}). One has $\chi(\E_l \otimes \E_l^*)=2(1- \rho(\A'_l)) \geq 4$ (see e.g. \cite[\S\;1]{L}). Furthermore, dualizing the sequence defining $\E_l$, we obtain
\[
\xymatrix{
          & 0    \ar[r] & H^0(\A'_l)^* \* \O_S \ar[r] & \E_l \ar[r]^{\hspace{-0.8cm} q} & \mathfrak{ext}^1_{\O_S} (\A'_l,\O_S) \ar[r] & 0.
}
\]
As  $h^0(\mathfrak{ext}^1_{\O_S} (\A'_l,\O_S))=h^1(\A'_l) >0$ (by \cite[Lemma 2.3]{G}), this proves that
$\E_l$ is globally generated off a finite set. 
Thus, as in the proof of \cite[Lemma 1.3]{L}, the linear system $|C|$ would contain a reducible curve,
a contradiction. This proves \eqref{eq:rho}.

The polynomial in $l$ in  \eqref{eq:rho} attains its minimum for $l_0=\frac{gr+r-d}{2r(d-r)}$.  The inequality
\eqref{eq:rho}  holds for the closest integer to $l_0$, which is $\alpha$. This proves \eqref{eq:bound}. 
\end{proof}

The property of $|H|$ not containing reducible curves is a dense, open property in $\K_p$. Therefore, the theorem proves the ``only if'' part in Theorem \ref{thm:main}(i).

\begin{remark} \label{rem:noexist2} 
(a) The proof does not require $\widetilde{C}$ to be smooth, only to possess a $g^r_d$. Thus the proof works for any {\it partial normalization} $\widetilde{C}$ of $C$ possessing a $g^r_d$, with $g:=p_a(\widetilde{C})$ instead of the geometric genus of $C$.

(b)
 Fix $p$,  $d$ and $r$. If  \eqref{eq:bound} holds for a given $\delta$, then it  holds for all $\delta' \geq \delta$. Indeed, for all positive integers $l$
and for any $\delta' \geq \delta$, we have 
$\rho(p,lr,ld+\delta') \geq \rho(p,lr,ld+\delta)$.
\end{remark}

Next we concentrate on the case $r=1$ and set $d=k$, where \eqref{eq:bound} 
reads like \eqref{eq:boundintro}. Then Theorem \ref{thm:noexist} proves the part of Theorem \ref{thm:main} stating that $V^k_{|H|,\delta}(S) \neq \emptyset$ only if \eqref{eq:boundintro} holds.

\begin{remark} \label{rem:noexist} Set $\rho=\rho(p,\alpha,k\alpha+\delta)$. 
It is convenient to write \eqref{eq:boundintro} in 
the form
 \begin{equation}\label{eq:del}
\rho \ge 0,\;\; \text{i.e.,}\;\; \delta \geq \frac{(g-k+1)^2-\beta^2}{4(k-1)},
 \end{equation}
where $\beta:=(k-1)(2\alpha+1) -g$, i.e., $-(k-1) < \beta  \leq k-1$.  
\end{remark}

As an aside, we obtain a bound on the  \emph{Beauville-Bogomolov self-intersection} of the rational curves in $\Hilb^k(S)$ corresponding  to curves in $V^k_{|H|,\delta}(S)$,  which confirms for them a conjecture by Hassett and Tschinkel \cite[Conj.~1.2]{HTint}. 

\begin{cor} \label{cor:bound} Let $(S,H)\in \K_p$  such that all elements in $|H|$ are reduced and irreducible (e.g., $\Pic (S) \cong \ZZ[H]$). Assume that $C \in V^k_{|H|,\delta}(S)$ is a curve whose normalization possesses a $g^1_k=|A|$ satisfying \eqref{eq:condnod1} and \eqref{eq:condnod2}. Let $R=R_{(C,A)}$ and $g:=p-\delta$.
Then
\begin{equation} \label{eq:boundqr}
 q(R)= 2(p-1) -\frac {{(g+k-1)}^ 2} {2(k-1)} = 2(\rho-1)- \frac{\beta^2}{2(k-1)} \geq -\frac{k+3}{2},
\end{equation}
with $\rho$ and $\beta$ as in Remark \ref {rem:noexist}.  
\end{cor}

\begin{proof}
The first equality on the left follows from \eqref{eq:qalfa2} and Lemma \ref{classR}, the middle equality is a direct computation and the inequality follows from Theorem \ref{thm:noexist} and Remark \ref{rem:noexist}. 
\end{proof}

\begin{remark} \label{rem:canskip}
  By Remarks \ref{rem:classR} and \ref{rem:noexist2}(a), the corollary also holds without assumptions \eqref{eq:condnod1} and \eqref{eq:condnod2} if we substitute the number $\delta$ with the number $\delta'$ of non-neutral nodes of $C$ with respect to the $g^1_k$. In particular, the inequality $q(R) \geq -\frac{k+3}{2}$ holds for any rational curve $R$ obtained from a curve in $V^k_{|H|,\delta}(S)$.
\end{remark}

\section{Chains of rational curves on unions of scrolls that are  limits of $K3$ surfaces} \label{ssec:deg}

We will henceforth fix an integer $p \geq 3$ and set $q=2p-1$.

\subsection{Unions of scrolls as limits of special $K3$ surfaces} \label{ssec:unionscr}
If $(S,L)\in \K_q$, then $\vert L\vert$ determines a morphism $\phi_{\vert L\vert}: S\to \PP^q$, which is an embedding for general
$(S,L)$. 
Let $\mathfrak{H}_q$ be the component of the Hilbert scheme of surfaces in
$\PP^q$ whose general point corresponds to an embedding of an $S$ as above. 
One has $\dim(\mathfrak{H}_q)=q^ 2+2q+19$ and $\mathfrak{H}_q$ is smooth at each point corresponding
to a smooth $K3$ surface.
The component  $\mathfrak{H}_q$ contains  points that correspond
to  \emph{degenerations} of elements of $\K_q$, as we will now explain (see \cite [\S 2.2] {clm}).  

Let $E \subset \PP^{q}$ be a smooth elliptic normal curve
of degree $q+1$ with two distinct line bundles $L_i \in \Pic^2(E)$, with $i=1,2$.
Let $R_1$ and $R_2$ be the rational normal scrolls of degree $q-1$ in
$\PP^{q}$ \emph{defined by} $L_1$ and $L_2$, respectively, i.e. $R_i$ is the union of lines
spanned by the divisors of  $\vert L_i\vert$. Then $R_1 \cap R_2 =E$,  the intersection is transversal and $E$ is anticanonical on each $R_i$. Moreover $R=R_1\cup R_2$ corresponds to  a smooth point of $\mathfrak{H}_q$. 

We will be concerned with the following case. Let $L_i \in \Pic^2(E)$, with $1\le i\le 2$, be two general line bundles. Consider the embedding of $E$ given by $L_2^ {\otimes p}=:\O_E(1)$. Then $R_1 \cong \PP^1 \x \PP^1$ and $R_2 \cong \mathbb{F}_2$. 
We let $\mathfrak{s}_i$ and $\mathfrak{f}_i$ denote the classes of the nonpositive section and fiber, respectively, of $R_i$, for $1\le i\le 2$. 
Then $\O_{R_1}(1) \cong \O_{R_1}(\mathfrak{s}_1 +(p-1)\mathfrak{f}_1)$ and 
$\O_{R_2}(1) \cong \O_{R_2}(\mathfrak{s}_2 +p\mathfrak{f}_2)$. The section $\mathfrak{s}_2$ does not intersect $E$, hence lies in the smooth locus of $R$, and it is embedded in $\PP^ q$ as a (degenerate) rational normal curve of degree $p-2$. In particular, $\mathfrak{s}_2$ is a Cartier divisor on $R$, so that 
\begin{equation}
  \label{eq:classnuova}
  H_0:=\O_R(1)\otimes \O_R(-\mathfrak{s}_2)
\end{equation}
is also Cartier, with $H_0^2=2p-2$. Moreover $\mathfrak{s}_2 \cdot H_0=p$. 

\begin{lemma} \label{lemma:mantengo}
There is a unique irreducible, codimension one subvariety 
$\mathfrak{H}'_q$ of $\mathfrak{H}_q$ containing the point corresponding to $R$ and smooth there, such that the general point of $\mathfrak{H}'_q$ represents a smooth $K3$ surface $S$ containing a smooth rational curve $\Gamma$ of degree $p-2$ degenerating to $\mathfrak{s}_2$ when $S$ flatly degenerates to $R$. The line bundle $H:=\O_{S}(1)\otimes \O_S(-\Gamma)$ is globally generated and primitive with ${H}^2=2p-2$. 
In particular, $(S,H)$ is a general element of  $\K'_p$ (cf. \S\;\ref{sub:k3}). 
\end{lemma}

\begin{proof}
  Let $\X \to \mathfrak{H}_q$ be the universal family and consider $\mathfrak{s}_2 \subset R_2 \subset R \subset \X$. Then $\mathfrak{s}_2$ stays off the singular locus of $\X$ and by standard deformation theoretic arguments (cf. e.g. \cite[II,Thm.~1.14]{Ko}), it moves inside $\X$  in a family $\F$ with 
\begin{equation}\label{eq:def}
 \dim(\F)\ge -K_{\X} \cdot \mathfrak{s}_2+\dim (\X)-3 = -K_{R_2} \cdot \mathfrak{s}_2 + \dim (\mathfrak{H}_q)-1 =  \dim (\mathfrak{H}_q)-1.
 \end{equation}
Since $\mathfrak{s}_2$ does not move on $R$ and since, for $S\in \mathfrak{H}_q$ outside a countable union of Noether-Lefschetz divisors, we have $\Pic(S)\cong \ZZ[L]$, with $L=\O_S(1)$, then equality must hold in \eqref {eq:def}. This implies that  $\F$  is smooth at the point corresponding to $\mathfrak{s}_2$. Hence  there is a  unique irreducible codimension one 
subvariety $\mathfrak{H}'_q$ in $\mathfrak{H}_q$, containing the point corresponding to $R$ and smooth there,  over which $\mathfrak{s}_2$ deforms.

Let $S$ be the surface corresponding to the general point of $\mathfrak{H}'_q$, let $L=\O_S(1)$  and let $\Gamma$ be the rational curve of degree $p-2$ that is a deformation of 
$\mathfrak{s}_2$.  Since the locus of pairs of scrolls inside $\mathfrak{H}_p$ has codimension $16$, the surface $S$ is irreducible with at most isolated double points of type $A_n$ for some $n\ge 1$, and $\Gamma$ sits in the smooth locus of $S$. Suppose $S$ is singular.
Then its minimal desingularization $\pi: S'\to S$ is a $K3$ surface and we set $L'=\pi^ *(L)$. By standard Hodge theory, the subvariety of $\mathfrak{H}_q$ corresponding to such singular surfaces $S$ is irreducible of codimension $1$, and all elements outside a countable union of Zariski closed proper subsets have one single double point of type $A_1$ (a \emph{node}) and $\Pic(S') \cong \ZZ[L']\oplus \ZZ[N]$, where $N$ is the $(-2)$--curve corresponding to the node. Hence such a singular $S$ cannot contain a smooth rational curve $\Gamma$ not containing the node. This proves that $S$ is smooth. 

To finish the proof, note that for $S$ general in  $\mathfrak{H}'_q$, the linear system  $|H|$ is base point free, as $|H_0|$ is,  and ${H}^2=2p-2$. Suppose $H=hA$, with $h>1$. Then $p-1=h^ 2(\gamma-1)$, where $A^ 2=2\gamma-2$. Moreover $p=H\cdot \Gamma=h(A\cdot \Gamma)$. Hence $h$ divides both $p-1$ and $p$, a contradiction. Thus  $H$ is indivisible, whence $(S,H)$ is general in $\K'_p$ by Proposition \ref {prop:ssub}.
\end{proof}

Let $\mathbb D$ be a disc. We fix $\varphi: \mathbb D \to \mathfrak{H}'_q$ a holomorphic map  with nonzero differential, such that $\varphi(0)$
is the point corresponding to $R$ and, for $t\in \mathbb D$ general, $\varphi(t)$ is a general element in $\mathfrak{H}'_q$.  By pulling back the universal family on $\mathfrak{H}_q$, we obtain a flat family $\X\to \mathbb D$, whose total space has isolated singularities along $E$ in the central fibre, and is otherwise smooth. Indeed, the singular locus is the zero locus of  the section in $H^0(E,T_{R}^1)$ corresponding to the section in $H^ 0(R,N_{R\vert \PP^ q})$ that is the image of the differential
of $\varphi$ at $0$ (cf. \cite[p.~647]{clm} or \cite[Sec. 3.1]{C}). Note that $T_{R}^1$ is a line bundle of degree $16$ on $E$ (cf. \cite[p.~644]{clm}).  

On  $\X$ we have  the pullback $\L$  of the hyperplane bundle on $\PP^q$.  The restriction to $\X$ of the  total space of the flat family $\F$ of deformations of $\mathfrak{s}_2$  is a surface $G$ contained in the smooth locus of $\X$. Hence it determines a line bundle $\G$ on $\X$. We set $\H=\L\otimes \G^ *$.
The restrictions of $\L$ and $\H$ to the general fibre $S$ of $\X\to \mathbb D$ give the
 two globally generated line bundles $L=\O_S(1)$ and $H$ 
 specializing to $\O_R(1)$ and $H_0$, respectively, on the central fiber $R$. The restriction of $G$ to the general fibre $S$ is the smooth rational curve $\Gamma$, specializing to $\mathfrak{s}_2$ on the central fibre, and  $L\cong \O_S(H+\Gamma)$. 
 
One can perform a small resolution of the singularities of $\X$, obtaining a new family $f: \mathcal S\to \mathbb D$, which has all properties indicated in \S \ref {ssec:USV}. The central fibre $S_0$ however is no longer $R$, but a modification of it. Precisely one can work things out in such a way that $S_0=R_1\cup \widetilde R_2$, where $\widetilde R_2$ is a sequence of blow ups of  
$R_2$ at the singular locus of $\X$, and $R_1$ and $\widetilde R_2$ meet transversally along $E\subset R_1$ and its strict transform (still denoted by $E$) on $\widetilde R_2$. 

Since the curves on $R$ we will be concerned with lie off the singular points of $\X$, we can and will work on $\X$ or $\mathcal{S}$ with no distinction.

\subsection{Special chains of rational curves}\label{sec:chains}

We now introduce the building blocks of 
limits on $R$ of  nodal hyperplane sections on the general surface in $\mathfrak{H}'_q$ in the degeneration described in \S\;\ref{ssec:unionscr}. 

Let $m$ be a positive integer satisfying $m \leq p$.
A {\it chain of length $2m-1$} is a
sum of $2m-1$ distinct lines
\[ f_{2,1}+f_{1,1}+f_{2,2}+f_{1,2} + \cdots + f_{2,m-1}+f_{1,m-1}+f_{2,m}, \; f_{i,j} \in |\mathfrak{f}_i|, \]
where 
$f_{1,j}$ intersects only $f_{2,j}$ and $f_{2,j+1}$, for $j=1, \ldots, m-1$.
The chain intersects $E$ in $2m$ points, consisting of $m$ divisors of  $|L_2|$, and 
$2m-2$ of them lie on the intersections between two lines, whereas the remaining two are on $f_{2,1}$ and $f_{2,m}$ and will be denoted by $a_1$ and $b_1$, respectively. This pair of points will be called {\it the distinguished pair of points of the chain}. 

We can also define chains with the roles of $R_1$ and $R_2$ interchanged, but we will not need this, except for the inductive argument in the proof of Lemma \ref{lemma:oddchain} right below. 

We will denote by $\C_{m}$ the family of  chains of length $2m-1$. Note that $\C_{m}$ is a locally closed subvariety of a Hilbert scheme of curves on $R$.

\begin{lemma} \label{lemma:oddchain}
  The map sending a  chain of length $2m-1$ 
to its pair of distinguished points on $E$ is a birational, injective morphism 
between $\C_{m}$ and 
$|mL_2-(m-1)L_{1}|$.
\end{lemma}

\begin{proof}
We describe the inverse map. Its existence is obvious if $m=1$. We proceed by induction on $m$. Let $a +b \in |mL_2-(m-1)L_{1}|$. Then $|L_2-a|=\{a'\}$ and $|L_2-b|=\{b'\}$. Therefore $a'+b' \sim (m-1)L_{1}-(m-2)L_2$ and we are done by induction, exchanging the roles of $L_1$ and $L_2$.
\end{proof}

We note that if $m \leq p$, then any chain of length $2m-1$ is
contained in a hyperplane. This fact will be used in the next section.

Assume we have  a chain of length $2m-1$ contained in a  hyperplane section $h$ of $R$. Let $\Gamma_1$ and $\Gamma_2$ be the sections of the rulings on $R_1$ and $R_2$ contained in $h$.  
 Then the distinguished points are $a_{1}:=\Gamma_{1} \cap f_{2,1}$ and 
$b_{1}:=\Gamma_{1} \cap f_{2,m}$.
We will call  $a_{1}+b_{1}$ {\it the $2$-cycle on $\Gamma_{1}$ associated to the chain}. We note that 
\begin{equation}
  \label{eq:impodd}
  a_{1}+b_1 \in {f}_*(|mL_{2}-(m-1)L_1|),
\end{equation}
where $f: E \to \Gamma_1 \cong \PP^1$ is the morphism determined by the linear 
series $|L_1|$.  
The chain intersects $\Gamma_1$ and $\Gamma_2$ in a total of $2m-1$ points, distributed as $m-1$ on $\Gamma_1$ and $m$ on $\Gamma_2$. 
They  will be called {\it the nodes of $h$ associated to the chain}. The nodes on $\Gamma_1$ will be called the {\it marked nodes} of $h$.
\begin{figure}[ht]
\[
\includegraphics[width=4cm]{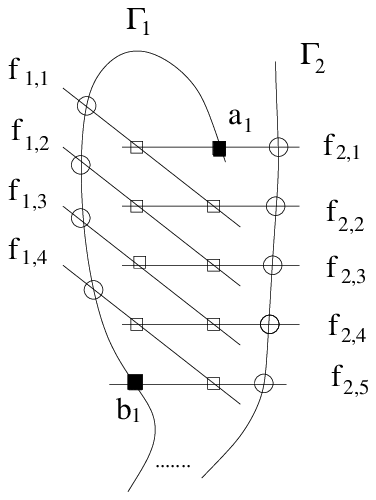}
\]
\caption{}
\label{fig:oddchain}
\end{figure}

Figure \ref{fig:oddchain} shows a chain of length $9$  contained in a hyperplane section.
All intersection points between the chain and $E$ are marked with a box, the two \emph{distinguished points} are marked with filled boxes, the  \emph{associated nodes} are marked with circles, the ones on $\Gamma_1$ are the \emph{marked nodes}.

In Figure \ref{fig:oddchain-norm} we describe the 
stable model of the partial normalization at the associated nodes of 
a hyperplane section containing  a chain: in the stable model all rulings are contracted, so that the two distinguished points  lying on 
$\Gamma_1$ are identified, creating a node.

\begin{figure}[ht]
\[
\includegraphics[width=12cm]{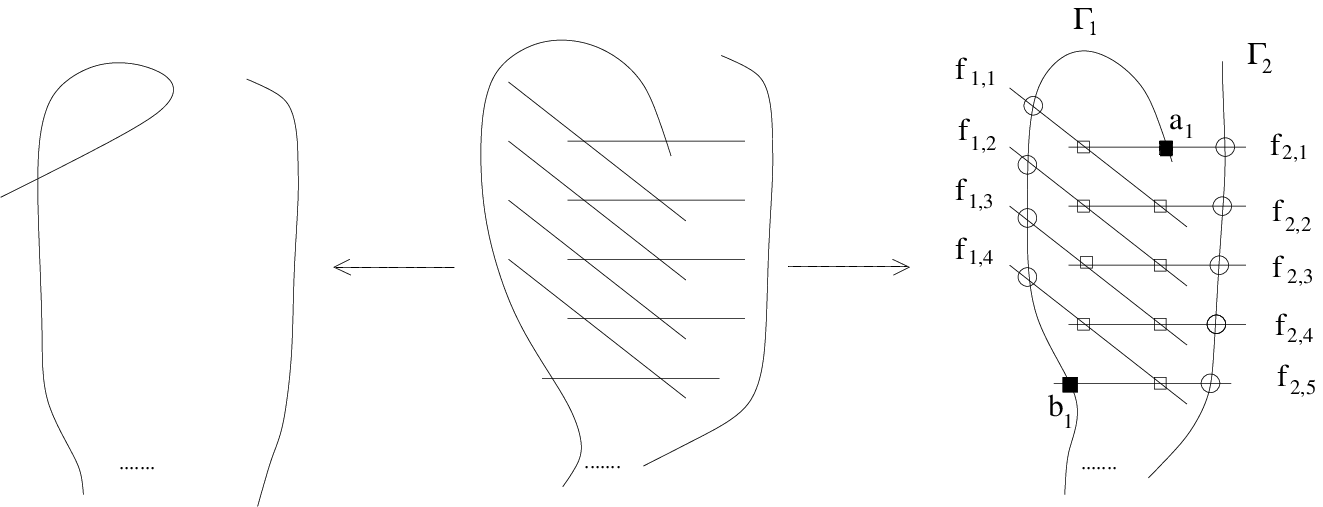}
\]
\caption{}
\label{fig:oddchain-norm}
\end{figure}


\section{Limits of nodal curves}\label{sec:limnod}


Let 
$ \alpha_{1}, \ldots,  \alpha_{p}$ be nonnegative integers such that 
\begin{equation}
  \label{eq:condex2}
  p=\displaystyle\sum_{j=1}^p j\alpha_j.
\end{equation}

\begin{definition} \label{def:v}
We define $V'(\alpha_{1}, \ldots,  \alpha_{p})$ to be the locally closed subset of
$|\O_R(1)|$ consisting of nodal curves $C$ not passing through the singular locus of $\X$ and containing exactly
$\alpha_{j}$  chains of length $2j-1$, for $j=1,\ldots,p$.  Condition \eqref{eq:condex2} implies that every element of $V'(\alpha_{1}, \ldots,  \alpha_{p})$ contains $p$ lines in $|\mathfrak{f}_2|$, thus it  contains
$\mathfrak{s}_2$. Hence the elements of $V'(\alpha_{1}, \ldots,  \alpha_{p})$ are in one-to-one correspondence with a locally closed subset of
$|H_0|$ (cf. \eqref{eq:classnuova}), which we denote by 
$V(\alpha_{1}, \ldots,  \alpha_{p})$. 
\end{definition}

Given a curve $C$ in $V(\alpha_{1}, \ldots,  \alpha_{p})$, we  denote by $\Gamma_1$ the section of the ruling on $R_1$ contained  in $C$ (i.e., $C$ is the union of $\Gamma_1$ and of  chains).
The curve $C$ comes  equipped with the subscheme of its
\begin{equation}
  \label{eq:numeronodi}
  \delta=\delta(\alpha_{1}, \ldots,  \alpha_{p}):= 
\displaystyle\sum_{j=1}^p (j-1) \alpha_{j} 
\end{equation}
\emph{marked nodes}  lying in the smooth locus of $R$. Set $g=p-\delta=  \sum_{j=1}^p\alpha_j$.

\begin{proposition}\label{prop:rightdim}  Under the condition \eqref{eq:condex2}, we have:
\begin{itemize}
\item [(i)] $V(\alpha_{1}, \ldots,  \alpha_{p})$ is a smooth component of $V_{|H_0|,\delta}(R)$ of the \emph{expected dimension} $g$;
\item [(ii)] $V(\alpha_{1}, \ldots,  \alpha_{p})$ is a component of the flat limit of $V_{\vert H\vert,\delta}(S)$ with
$(S,H)\in \K'_p$ general.
\end{itemize}
\end{proposition}

\begin{proof} By Lemma \ref{lemma:oddchain} we have a morphism  
\[
\xymatrix{
\nu: V(\alpha_{1}, \ldots,  \alpha_{p}) \ar[r] &  
\displaystyle\prod_{j=1}^p \Sym^{\alpha_j}|jL_2-(j-1)L_{1}|
} 
\]
The target is irreducible of dimension $g$, by \eqref{eq:condex2} and \eqref{eq:numeronodi}. Take any point $\eta$ therein.
By Lemma \ref{lemma:oddchain} each coordinate of $\eta$  uniquely defines  a chain. The reduced intersection between the union of these chains  and $E$ is an effective divisor  $D_{\eta}$  lying in 
$|(\sum j\alpha_j)L_2|=|pL_2|=|\O_E(1)|$, by \eqref{eq:condex2}. Hence $\nu$ is dominant and any fiber of $\nu$ is a point. Therefore $V(\alpha_{1}, \ldots,  \alpha_{p})$ is irreducible of dimension $g$. 

 Since the $\delta$ nodes are not disconnecting  (cf. \S\;\ref{sec:chains}), the tangent space to $V_{|H_0|,\delta}(R)$ at a point of $V(\alpha_{1}, \ldots,  \alpha_{p})$ has dimension $g$.
Hence (i) follows.  Assertion (ii) follows by Lemma \ref {lem:defo}.  
\end{proof}

The next result follows from the description of the stable model of the partial normalization at the associated nodes of a hyperplane section containing a chain in \S\;\ref{sec:chains} and the fact that $C$ is obtained by removing the component $\mathfrak{s}_2$ from a hyperplane section of $R$. 

\begin{lemma} \label{lemma:stablemodnorm}
  Let $C$ be a curve in $V(\alpha_{1}, \ldots,  \alpha_{p})$ and $\overline{C}$ the stable model of its partial normalization at its $\delta$ marked nodes. Then
 $\overline{C}$ is the image of $\Gamma_1 \cong \PP^1$ by the morphism identifying each distinguished pair of points of each  chain contained in $C$. 
Thus  $\overline{C}$ has arithmetic genus $g$. 
\end{lemma}

The case $\delta=0$ corresponds to
$(\alpha_{1}, \ldots,  \alpha_{p})=(p,0,\ldots,0)$, and then $\overline{C}$ has $p$ nodes. Then, in the degeneration of a general $(S,H)\in \K'_p$ to $(R,H_0)$,  the  general element in $|H|$ degenerates  to an irreducible $p$--nodal rational nodal curve. 


\section{Limits of $k$-gonal nodal curves} \label{S:kn}


Let $k \geq 2$ be an integer. We keep the notation introduced in \S \ref {sec:limnod}.

\subsection{$k$-gonal nodal curves in the central fibre} 

\begin{definition}\label{defn:limkgonal}  
We define $V^k(\alpha_1,\ldots,\alpha_p)$ to be 
the closed subset of  $V(\alpha_1,\ldots,\alpha_p)$
consisting of curves $C$ such that all  $2$--cycles (in number of $g$) on $\Gamma_1$ associated to a chain belong to divisors of the same $g^1_{k}$ on $\Gamma_1$.
\end{definition}

The following is a consequence of Lemma \ref{lemma:stablemodnorm}.

\begin{lemma} \label{lemma:kgonale}
 Let $C$ be a curve in $V( \alpha_1,\ldots,\alpha_p)$.  The stable model of the partial normalization of $C$ at its $\delta$ marked nodes lies in $\overline{\M^1_{g,k}}$ if and only if $C$ lies in $V^k(\alpha_1,\ldots,\alpha_p)$.  In particular  $V^k( \alpha_1,\ldots,\alpha_p)$  fills up one or more components of
$V^k_{|H_0|,\delta}(R)$.
\end{lemma}

 If $g\leq 2(k-1)$, then $V^k( \alpha_1,\ldots,\alpha_p)=V( \alpha_1,\ldots,\alpha_p)$ by Lemma \ref{lemma:stablemodnorm}. As we will see in Proposition \ref {prop:riempie} below, this also follows  by a parameter count concerning $g^ 1_k$'s on $\PP^ 1$.  By the  argument as in the proof of Proposition \ref{degarg} (but also by the same parameter count as above), the expected dimension of $V^k( \alpha_1,\ldots,\alpha_p)$ is $2(k-1)$ when $g > 2(k-1)$. Our objective is to prove that, under suitable conditions, 
$V^k( \alpha_1,\ldots,\alpha_p)$ is nonempty of the expected dimension. To do so, we need some intermediate  results.

\subsection{Some technical results}\label{ssec:tech} Recall the morphism $f:  E \rightarrow  \PP^1$ determined  by  $\vert L_1\vert$. We have  the induced map
$f^{(2)}: \Sym^2(E) \rightarrow \Sym^2(\PP^1)$.

As customary, we identify $\Sym^2(\PP^1)$ with  $\PP^2$:  fix an irreducible conic  $\Delta\cong \PP^ 1$, and identify a divisor $x+y$ of $\Delta$ with:
\begin{itemize}
\item  the pole of the line $\langle x,y\rangle$ with respect to $\Delta$, if $x\neq y$;
\item the point $x\in \Delta$, if $x=y$.
\end{itemize}
In this way $\Delta$ is identified with the \emph{diagonal} of $\Sym^2(\PP^1)$ and the \emph{coordinate curve} $\{x+y, y\in \PP^ 1\}$ with the tangent line $\ell_x$ to $\Delta$ at $x$.

We denote by  $\mathfrak{c}_{j}$, for $1\le j \le n$, the image via $f^{(2)}$ of the smooth rational curves in $\Sym^2(E)$ defined by the 
pencils $|jL_{2}-(j-1)L_{1}|$. These are distinct conics, each intersecting $\Delta$ in four distinct points corresponding to the ramification points of the pencils.

Consider $\mathbb Q=\PP^ 1\times \PP^ 1$ with the two projections $\pi_i: \mathbb Q\to \PP^ 1$, $i=1,2$.  Fix a positive integer $k$ and look at the line bundle $\O_\mathbb Q(k,k):=p_1^ *(\O_{\PP^ 1}(k))\otimes p_2^ *(\O_{\PP^ 1}(k))$, whose space of sections is $H^ 0(\PP^ 1, \O_{\PP^ 1}(k))^ {\otimes 2}$. The two subspaces $\Sym^ 2(H^ 0(\PP^ 1, \O_{\PP^ 1}(k)))$ and 
$\wedge^ 2 H^ 0(\PP^ 1, \O_{\PP^ 1}(k))$ are invariant  (resp. anti-invariant) under the natural involution that exchanges the coordinates. Hence they are pull-backs of sections of line bundles, $\mathcal O_k^ +$ and $\mathcal O_k^ -$ respectively, on $\Sym^ 2(\PP^ 1)$. 

Let us focus on  $\mathcal O_k^ -$. One has 
\[H^ 0(\Sym^ 2(\PP^ 1), \mathcal O_k^ -)\cong \wedge^ 2 H^ 0(\PP^ 1, \O_{\PP^ 1}(k))\cong
H^ 0(\PP^2, \mathcal O_{\PP^ 2}(k-1)).\]
Therefore the linear system
$\vert \mathcal O_k^ -\vert$ identifies with $\vert \mathcal O_{\PP^ 2}(k-1))\vert$  and also with $\PP( \wedge^ 2 H^ 0(\PP^ 1, \O_{\PP^ 1}(k)))$.
Under the former isomorphism,
a point  $\mathfrak{g}$ of the grassmannian $\mathbb G(1,k)\subset \PP( \wedge^ 2 H^ 0(\PP^ 1, \O_{\PP^ 1}(k)))$, which can be identified with a linear series $g^1_k$ on $\PP^1$,  corresponds to the degree $k-1$ curve in $\PP^ 2$
\[ C_{\mathfrak{g}}= \{ W \in \Sym^2(\PP^1) \; | \; \mathfrak{g}(-W) \geq 0\}.\]
The family of curves $\{C_\mathfrak g\}_{\mathfrak g\in \mathbb G(1,k)}$  is irreducible of dimension $2(k-1)=\dim(\mathbb G(1,k))$.

\begin{lemma} \label{lemma:trans} For general choices of  $L_1$, we have:
\begin{itemize}
\item [(i)] no curve $\mathfrak{c}_{j}$ is contained in a curve $C_{\mathfrak{g}}$, for  $\mathfrak{g}=g^1_k$ on $\PP^1$;

\item [(ii)]  for a  general  $\mathfrak{g}$, the curve 
$C_{\mathfrak{g}}$ intersects each $\mathfrak{c}_{j}$ transversally in $2(k-1)$ distinct points;

\item [(iii)]  in addition, none of these $2(k-1)$ points  is fixed varying $\mathfrak g$. 
\end{itemize}\end{lemma}

\begin{proof} By moving $L_{1}$, each 
conic $\mathfrak c_{j}$ moves in $\PP^ 2$ in a $1$-dimensional family containing the diagonal $\Delta$ (obtained for $L_1=L_2$). But $\Delta$ is not contained in any $C_\mathfrak g$, proving (i).  Similarly, it suffices to prove (ii) for the intersection with $\Delta$ of a $C_{\mathfrak{g}}$, with  $\mathfrak{g}$ general, which is obvious, since the intersection points correspond to the ramification points of ${\mathfrak{g}}$.  Part (iii) easily follows. 
 \end{proof}

Lemma \ref{lemma:oddchain} can be rephrased as the following lemma, whose proof  is left to the reader. 

\begin{lemma} \label{lemma:coppie}
A chain of length $2m-1$ determines 
and is determined by a point on the conic $\mathfrak{c}_{m}$. 
\end{lemma}

\subsection{Nonemptiness and dimension of limit $k$-gonal nodal curves systems}  \label{nonempty} 

We can now prove the desired result about nonemptiness and dimension
of $V^k( \alpha_1,\ldots,\alpha_p)$.

\begin{prop} \label{prop:riempie}
If   \eqref{eq:condex2} and
\begin{equation}
\label{eq:condex3} \alpha_{j} \leq 2(k-1) \; \; \mbox{for all} \; \; j 
\end{equation}
hold, the variety 
$V^k( \alpha_1,\ldots,\alpha_p)$
is nonempty of the expected dimension $\min\{2(k-1),g\}$  and is a component
of
${V^k_{|H_0|,\delta}(R)}$.  Furthermore, for the general curve in $V^k( \alpha_1,\ldots,\alpha_p)$, the family of $g^1_k$'s on $\Gamma_1$ satisfying the condition in Definition \ref{defn:limkgonal} has dimension $\max\{0,\rho(g,1,k)\}$. 
\end{prop}

\begin{proof}
Consider the set of curves $\{C_\mathfrak g\}_{\mathfrak g\in \mathbb G(1,k)}$ of dimension $2(k-1)$ above, which is in one-to-one correspondence with the set of 
$g^1_k$'s on $\Gamma_1\cong \PP^1$. A general $C_{\mathfrak{g}}$ intersects each conic $\mathfrak{c}_{j}$ in $2(k-1)$ distinct points by (ii) of Lemma \ref {lemma:trans}.  By   \eqref{eq:condex3},  we can pick  $\alpha_{j}$ of them.  
By (iii) of Lemma \ref {lemma:trans},   each of these $\alpha_{j}$ intersection points gives rise to a chain of length $2j-1$ not intersecting the singular locus of $\X$ with distinguished pair given by the
$2$-cycle corresponding to the point itself.  Then the construction in (the proof of) Proposition \ref{prop:rightdim} yields the existence of curves in 
$V( \alpha_1,\ldots,\alpha_p)$ containing all the above chains. 
These curves lie in $V^k( \alpha_1,\ldots,\alpha_p)$ by construction.

When $g >2(k-1)$, this shows that $V^k( \alpha_1,\ldots,\alpha_p)$ is irreducible (and nonempty) of dimension $2(k-1)$ and that its general element corresponds to finitely many of the curves $C_{\mathfrak{g}}$. When $g \leq 2(k-1)$, one has $V^k( \alpha_1,\ldots,\alpha_p)=
V( \alpha_1,\ldots,\alpha_p)$ and its general element corresponds to a family of dimension $2(k-1)-\sum\alpha_{j}=2(k-1)-g=\rho(g,1,k)$ of the curves $C_{\mathfrak{g}}$. The fact that $V^k( \alpha_1,\ldots,\alpha_p)$ is a component
of
${V^k_{|H_0|,\delta}(R)}$ follows from Lemma \ref{lemma:kgonale}.
\end{proof}

\begin{cor} \label{rem:ram2} Assume that  \eqref{eq:condex2} and  \eqref{eq:condex3} hold. One has:
\begin{itemize}
\item [(i)] if $(S,H)\in \K'_p$ is general, then ${V^k_{|H|,\delta}(S)}$ has a component $V$ of the expected dimension $\min\{2(k-1),g\}$  whose limit when $S$ tends to $R$ contains $V^k( \alpha_1,\ldots,\alpha_p)$ as a component;
\item [(ii)] the normalization $\widetilde{C}$ of the general curve $C$ in $V$ satisfies $\dim (W^1_k(\widetilde{C}))=\max\{0,\rho(g,1,k)\}$;
\item [(iii)]  the general  $g^1_k$ in any component of $W^1_k(\widetilde{C})$ has simple ramification;
\item [(iv)] the $\delta$ nodes of $C$ are non-neutral with respect to the general  $g^1_k$ in any component of $W^1_k(\widetilde{C})$.
\end{itemize}
\end{cor} 

\begin{proof} By Proposition \ref {prop:riempie}, $V^k( \alpha_1,\ldots,\alpha_p)$ is not empty, of the expected dimension $\min\{2(k-1),g\}$. Then (i) follows by 
Proposition \ref{degargrel} (note that the  nodes on the singular locus of $R$  of the general curve in $V^k( \alpha_1,\ldots,\alpha_p)$ must smooth by well-known arguments). The properties stated at the end of Proposition \ref{prop:riempie} imply (ii). 

Let us prove  (iii). The stable model of a general curve in $V^k( \alpha_1,\ldots,\alpha_p)$ is the image of $\PP^1$ by a morphism identifying $g$ pairs of points in a general $\mathfrak g=g^ 1_k$ on $\PP^ 1$ (see 
Lemma \ref {lemma:kgonale}).  The limit ramification points of a general $g^ 1_k$ on $\widetilde{C}$ are 
 the ones of  $\mathfrak g$ plus the ones tending to the nodes of $\overline C$.  The former ramification is simple by the generality of $\mathfrak g$.  The latter is simple because, in the admissible cover setting,  each node is replaced
by a $\PP^ 1$ joining the two branches and mapping $2:1$ to a $\PP^ 1$, hence the ramification is simple there. 

Finally (iv) holds because the same happens for the $\delta$ marked nodes on a general limit curve in $V^k(\alpha_1,\ldots,\alpha_p)$ and the related admissible cover.  \end{proof}


\section{Conclusion of the proof of the main theorem} \label{sec:finalexistence}


In this section we finish the proof of  Theorem \ref{thm:main}. We can assume that $\delta \leq p-1$ and $p \geq 3$, as otherwise the result is trivial. 
The theorem follows first for the general polarized surface $(S,H) \in \K'_p$ by   Corollary \ref{rem:ram2} and  Proposition \ref{lemma:trovainteri} right below. As is well-known, the general $(S,H) \in \K'_p$ is a limit of a one-parameter family of general polarized surfaces in $\K_p$. Hence the theorem follows by applying Proposition \ref{degargrel} once more to this latter family.

\begin{proposition} \label{lemma:trovainteri}
  Let $p,k,\delta$ be integers satisfying $p \geq 3$, $k \geq 2$, 
$\delta \leq p-1$ and  \eqref{eq:boundintro}. 
Then there are integers $\alpha_{j}$ satisfying 
\eqref{eq:condex2}-\eqref{eq:condex3}  with 
$\delta(\alpha_1,\ldots,\alpha_p)=\delta$.
\end{proposition}

To prove this, we need two auxiliary results. 

\begin{lemma} \label{lemma:bastaquellominimale}
  Assume there are integers $\alpha_{j}$ satisfying 
\eqref{eq:condex2}-\eqref{eq:condex3}  with 
$\delta_0=\delta(\alpha_1,\ldots,\alpha_p)$. Then, for any integer $\delta'$ satisfying
\[ \delta_0 \leq \delta' \leq p-1 \]
there are integers $\alpha'_{j}$ satisfying 
\eqref{eq:condex2}-\eqref{eq:condex3}  with 
$\delta(\alpha'_1,\ldots,\alpha'_p)=\delta'$.
\end{lemma}

\begin{proof} Assume $\delta_0<p-1$.
It suffices to find integers $\alpha'_{j}$ satisfying 
\eqref{eq:condex2}-\eqref{eq:condex3}  with 
$\delta(\alpha'_1,\ldots,\alpha'_p)=\delta_0+1$. 
To do this, we pick the two largest indices $j_1,j_2$, possibly coinciding, for which  
$\alpha_{j_1}+\alpha_{j_2} \geq 2$. These indices do exist, otherwise 
$\alpha_{j_0}=1$ for one index $j_0$ and 
$\alpha_j =0$ for $j \neq j_0$. Then $p=j_0$ by \eqref{eq:condex2} and $\delta_0=j_0-1$ by \eqref{eq:numeronodi}, so that $\delta_0=p-1$, a contradiction.

If $j_1=j_2$, we define  
\[ \alpha'_{j}=  
\begin{cases}  
\alpha_j-2 &  \; \mbox{if} \; j=j_1, \\ 
1  & \;\mbox{if} \; j=2j_1, \\
\alpha_{j} & \;\mbox{otherwise}.
\end{cases} 
\]
If $j_1\neq j_2$, we define  
\[ \alpha'_{j}=  
\begin{cases}  
\alpha_j-1 &  \; \mbox{if} \; j=j_1 \;\mbox{or} \; j=j_2.  \\ 
1  & \;\mbox{if} \; j=j_1+j_2, \\
\alpha_{j} & \;\mbox{otherwise}.
\end{cases} 
\]
These integers satisfy the desired conditions.
\end{proof}

Define now
\begin{equation}
  \label{eq:defm}
 m= m(p,k):=\max\{n \in \ZZ \; | \; (k-1)n(n+1) \leq p\} 
\end{equation}
and
\begin{equation}
  \label{eq:deft}
t=t(p,k):= \max\{n \in \ZZ \; | \; (k-1)m(m+1)+n(m+1) \leq p\},  \;\text {i.e.}\; t= \Big\lfloor \frac{p}{m+1}\Big\rfloor-m(k-1).
\end{equation}
We have a unique representation
 \begin{equation} \label{eq:propp}
p = (k-1)m(m+1) + t(m+1) +\lambda.
\end{equation} 
Note that
\begin{equation} 
\label{eq:propl} 0 \leq t < 2(k-1) \; \;\mbox{and} \; \;
0  \leq  \lambda  \leq m. 
\end{equation}

\begin{lemma} \label{lemma:deltaminimo}
Given $p$ and $k$,  the minimal integer $\delta$ satisfying \eqref{eq:boundintro} is
 \begin{equation}
    \label{eq:delta-algo} 
    \delta_0=\delta_0(p,k)=(k-1)m(m-1)+tm+\lambda= \Big\lceil \frac{mp}{m+1}\Big\rceil-m(k-1).
    \end{equation}
\end{lemma}

\begin{proof}
Straightforward computation using \eqref{eq:propl}. 
\end{proof}

\renewcommand{\proofname}{Proof of Proposition \ref{lemma:trovainteri}}

\begin{proof}
 We will find integers $\alpha_{j}$ satisfying 
\eqref{eq:condex2}-\eqref{eq:condex3}  with 
$\delta(\alpha_1,\ldots,\alpha_p)=\delta_0$ as in Lemma \ref{lemma:deltaminimo}. The result then follows from Lemma \ref{lemma:bastaquellominimale}.
If $\lambda=0$, we let  
\[ \alpha_{j}=  
\begin{cases}  
2(k-1) & \;\mbox{for} \; j=1,\ldots,m; \\
t & \;\mbox{for} \; j=m+1; \\
0 & \;\mbox{for} \; j >m+1.
\end{cases} 
\]
If $t =0$ and $\lambda >0$, we let 
\[ \alpha_{j}=  
\begin{cases}  
2(k-1) & \;\mbox{for} \; j=1,\ldots,m-1;\\
2(k-1)-1 & \;\mbox{for} \; j=m;\\
1 & \;\mbox{for} \; j=m+\lambda;\\
0 & \;\mbox{otherwise}.
\end{cases} 
\]
Finally, if $t>0$ and $\lambda >0$, we let 
\[ \alpha_{j}=  
\begin{cases}  
2(k-1) & \;\mbox{for} \; j=1,\ldots,m;\\
t-1 & \;\mbox{for} \; j=m+1;\\
1 & \;\mbox{for} \; j=m+1+\lambda;\\
0 & \;\mbox{otherwise}.
\end{cases} 
\]
Then \eqref{eq:condex2}-\eqref{eq:condex3} are verified with
$\delta(\alpha_1,\ldots,\alpha_p)= \delta_0$.
\end{proof}

\renewcommand{\proofname}{Proof}

\section{The Mori cone of  punctual Hilbert schemes and related conjectures} \label{S:ratcur2}

Finally we go back to the topic of  \S\;\ref{S:ratcur} (from  which we keep the notation).
Let $(S,H) \in \K_p$. We denote by $R_{p,\delta,k}$ the  rational curves (or their classes) in $\Hilb^k(S)$ associated to the $g^1_k$  on the normalizations of curves in $V^k_{|H|,\delta}(S)$ satisfying \eqref{eq:condnod1} and \eqref{eq:condnod2}. Theorem \ref{thm:main} determines the triples $(p,k,\delta)$ for which the rational curves $R_{p,\delta,k}$ exist and, in these cases, yields the existence of an irreducible family  of such curves of  dimension exactly $2(k-1)$, which is the expected dimension of any family of rational curves on a hyperk{\"a}hler manifold of dimension $2k$ (cf. \cite[Cor.~5.1]{ran}). 
By Lemma \ref{classR}, the class of $R_{p,\delta,k}$ in 
$N_1(\Hilb^k(S))$ is $H-(p-\delta+k-1)\mathfrak{r}_k$. Given $p$ and $k$, the class $R_{p,\delta,k}$ that is the closest to the border of the Mori cone is the one corresponding to $\delta=\delta_0$ minimal satisfying \eqref{eq:boundintro}, given by \eqref {eq:delta-algo}. We call such a curve, or class, \emph{optimal}. Thus:

\begin{prop} \label{prop:optimalclass}
  For fixed $p$ and $k$, the optimal class is 
  \begin{equation}
    \label{eq:optimalclass}
   R_{p,\delta_0,k} \equiv  H-\Big( (m+1)(k-1)+ \Big\lfloor \frac{p}{m+1} \Big\rfloor\Big)\mathfrak{r}_k, 
  \end{equation}
where $m$ is as in \eqref{eq:defm}.
\end{prop}

\begin{remark} \label{rem:optimal}
  As $\rho(p,l,kl+\delta-1)=\rho(p,l,kl+\delta)-l-1$ for any positive integer $l$, the curve $R_{p,\delta,k}$
is optimal if and only if $\rho\leq \alpha$, where $\rho$ is as in Remark \ref{rem:noexist} and $\alpha$ as in \eqref{eq:boundintro}. This follows from
Theorem \ref{thm:main}.
\end{remark}

A result by  Huybrechts \cite[Prop. 3.2]{H2} (resp. Boucksom \cite{Bou}) says that a divisor 
$D$ on a hyperk{\"a}hler manifold $X$ is nef (resp. ample)
if and only if $q(D)\ge 0$ and $D \cdot R \geq 0$ (resp. $q(D)> 0$ and $D \cdot R > 0$) for any (possibly
singular) {rational} curve
$R\subset X$. Theorem \ref{thm:main} allows a small step towards the determination of the ample (or nef) cone:

\begin{prop} \label{prop:conoampio}
 Let $(S,H)\in \mathcal K_p$. If the $\mathbb{Q}$-divisor $D=H-t\mathfrak{e}_k$ in $\Hilb^k(S)$ is ample (resp. nef), then 
  \begin{equation}
   \label{eq:conoampio}
 0 < t <  \tau(p,k):= \frac{2(p-1)}{(m+1)(k-1)+ \Big\lfloor \frac{p}{m+1} \Big\rfloor} \; \; \mbox{(resp. \; $0 \le  t \le  \tau(p,k)$ )},  
 \end{equation}
where $m$ is as in \eqref{eq:defm}.
\end{prop}

\begin{proof}
 We may assume that $(S,H)$ is general. By Theorem \ref{thm:main}, the optimal class $R_{p,\delta_0,k}$ as in \eqref{eq:optimalclass} is effective, hence $D \cdot R_{p,\delta_0,k} >0$ (resp. $\geq 0$), which is equivalent to the right hand inequality in \eqref{eq:conoampio}. The other inequality follows from $D \cdot \mathfrak{r}_k >0$ (resp. $\geq 0$). 
\end{proof}

Since $N_1(\Hilb^k(S))$ has rank two when $\Pic S \cong \ZZ[H]$, it is natural to pose the following:

\begin{question} \label{cong2}
Let $(S,H)\in \mathcal K_p$ be general, with $\Pic S \cong \ZZ[H]$ and $p \geq 2$.
Are the extremal rays of 
the Mori cone of $\Hilb^k(S)$ generated by $\mathfrak{r}_k$ and by the optimal class $R_{p,\delta_0,k}$ (cf. \eqref{eq:optimalclass})?
\end{question}

An affirmative answer to this question would yield that the bound
\eqref{eq:conoampio} is optimal. We will see below (Corollary
\ref{cor:intminima}, Propositions \ref{prop:BM} and \ref{prop:intzero2}) that
the question has an affirmative answer for
infinitely many pairs $(p,k)$. However, the question does not have a
positive answer for all pairs $(p,k)$. Indeed,
as communicated to us by Hassett, one can show that the rational
curves in \cite[Exmpl.~7.7]{HT}, where $p=8$ and $k=2$, have class
$3H-16\mathfrak{r}_k$, whereas the optimal class is
$H-5\mathfrak{r}_k$. Another example with $k=2$ yielding a negative
answer to Question \ref{cong2} was provided by Bayer and Macr{\`i} in
\cite[Exmpl.~13.4]{BM-MMP} after the appearance of the first version of
this paper on the web. In fact, in their work, Bayer and Macr{\`i}
determine the generators of the Mori cone of $\Hilb^k(S)$, up to
solving some Pell's equations. A substantial step towards determining the
Mori cone by our approach would be to extend our analysis to the
nonprimitive case, i.e. to curves in $|nH|$ for $n>1$. The methods of
our paper should in principle  work to treat this more general
case, although the extension is nontrivial. We plan to return to this
in future research.

\subsection{Conjectures of Hassett and Tschinkel}\label{ssec:HT}   Hassett and Tschinkel conjecture in  \cite[Conj. 1.2]{HTint}  that, for any polarized variety $(X,\mathfrak{g})$ deformation equivalent to $\Hilb^k(S)$ with  $S$ a $K3$ surface,  a $1$-cycle $R$ is effective if and only if $R \cdot \mathfrak{g} >0$ and $q(R) \geq -(k+3)/2$.  The ``only if'' part was proved for $k=2$ in \cite{HT2}, and then
for all $k \geq 2$ by Bayer and Macr{\`i} in \cite[Prop.~12.6]{BM-MMP}
after the submission of this paper. Hassett and
Tschinkel also conjecture in
\cite[Thesis 1.1]{HTint} that $-(k+3)/2$ is the self-intersection of the lines in any $\PP^k \subset X$. This latter conjecture has been verified  for $k=2$ in \cite{HT2}, $k=3$ in \cite{HHT} and $k=4$ in \cite{BJ}.  Other self-intersections have different geometrical properties from the point of view of birational geometry: primitive generators $R$ of extremal rays on a hyperk\"ahler manifold such that the associated extremal contraction is divisorial must satisfy
$-2 \leq q(R) <0$ by \cite[Thm.~2.1]{HTint}.

As noted in \S~\ref{S:nonex}, our Corollary \ref {cor:bound} proves
(and gives a Brill-Noether theoretical interpretation of) the inequality
$q(R) \geq -\frac{k+3}{2}$ for any rational curve arising from a
$g^1_k$ on the normalization of a nodal curve in $|H|$ satisfying
\eqref{eq:condnod1} and \eqref{eq:condnod2}, where $(S,H)\in \mathcal K_p$ is such that $|H|$ contains no reducible curves. (In fact, the conditions \eqref{eq:condnod1} and \eqref{eq:condnod2} can be skipped, by Remark \ref{rem:canskip}.)
For the curves $R_{p,\delta,k}$ obtained from Theorem \ref{thm:main}, we have
by Corollary \ref{cor:bound},
\begin{equation} \label{eq:selfint}
q(R_{p,\delta,k})=2(p-1) -\frac{(p-\delta+k-1)^2}{2(k-1)}=2(\rho-1)-\frac{\beta^2}{2(k-1)},
\end{equation}
with $\rho$ as in Remark \ref{rem:noexist} and  $\alpha$ as in \eqref{eq:boundintro}. (Recall that $\rho \geq 0$ and $-(k-1) < \beta \leq k-1$.) We can deduce the existence of the curves with lowest self-intersection in 
Hassett-Tschinkel's conjecture:

\begin{prop} \label{prop:intminima}
 Let $(S,H)\in \mathcal K_p$ be general and $R_{p,\delta,k}$ as above.
Then $q(R_{p,\delta,k})=-\frac{k+3}{2}$ if and only if $p=s(s+1)(k-1)$ for an integer $s \geq 1$ and $R_{p,\delta,k}$ is optimal. 
\end{prop}

\begin{proof}
  By \eqref{eq:selfint}, we have $q(R_{p,\delta,k})=-\frac{k+3}{2}$ if and only if
$\rho=0$ and $\beta=k-1$, in which case $R_{p,\delta,k}$ is optimal. A straightforward computation yields $g=2\alpha(k-1)$ and $p=\alpha(\alpha+1)(k-1)$.
Conversely, if $p$ is of the given form, 
one checks that $\delta=p-2s(k-1)$ verifies \eqref {eq:boundintro} and $q(R_{p,\delta,k})=-\frac{k+3}{2}$. 
\end{proof}

 If $k=2$, the self-intersection $-5/2$ is obtained if and only if $p=s(s+1)$ for an integer $s \geq 1$, which are  the genera covered in \cite[Prop.~7.2]{fkp2}, where rational curves with self-intersections $-5/2$ in $\Hilb^2(S)$ covering a $\PP^2$ were obtained  by  different methods.

Since $N_1(\Hilb^k(S))$ has rank two when $\Pic S \cong \ZZ[H]$, we have, as a consequence of the recent
result \cite[Prop.~12.6]{BM-MMP}:

\begin{cor} \label{cor:intminima}
Let $(S,H) \in \K_p$ be general such that $\Pic S \cong \ZZ[H]$ and $p=s(s+1)(k-1)$ for an
integer $s \geq 1$. Then the extremal rays of
the Mori cone of $\Hilb^k(S)$ are generated by $\mathfrak{r}_k$ and by
the optimal
class $R_{p,p-2s(k-1),k} \equiv H-(2s+1)(k-1)\mathfrak{r}_k$.
\end{cor}

After the appearance of a first version of this paper on the web, Bayer and Macr{\`i} \cite[Prop.~9.3]{BM} proved that Question
\ref{cong2} also has an affirmative answer in the ``simplest'' cases $p \leq 2(k-1)$, that is, when $\delta_0=0$. More precisely, in our notation:

\begin{proposition} \label{prop:BM}
{\rm (Bayer and Macr{\`i})} Let $(S,H) \in \K_p$ be general and $k \geq 2$ be an integer such that $p \leq 2(k-1)$. Then the extremal rays of 
the Mori cone of $\Hilb^k(S)$ are generated by $\mathfrak{r}_k$ and by  the optimal 
class $R_{p,0,k} \equiv H-(p+k-1)\mathfrak{r}_k$.
\end{proposition}

In 
 \cite[Rem.~9.4]{BM} one also proves that the class $\overline{R}:=R_{p,0,k}-\mathfrak{r}_k$ has positive intersection with an ample class and $q(\overline{R}) \geq -\frac{k+3}{2}$ for $k >>0$, proving that the Mori cone is strictly smaller than predicted by Hassett and Tschinkel in \cite[Conj. 1.2]{HTint}. In Remark \ref{rem:HT} below we provide one more series of cases where the same phenomenon occurs. This shows that the ``if'' part of their conjecture needs some modification concerning the minimal Beauville-Bogomolov self-intersection of the generators. 

For fixed $k$, the possible values of $q(R_{p,\delta,k})$ are a set of rational numbers of the form on the right hand side of \eqref{eq:selfint}. In the optimal case we find that $\beta=k-1-t$ and $\rho=\lambda$, where $t$ and $\lambda$ are as in \eqref{eq:deft} and \eqref{eq:propp}, that is
\begin{equation}
  \label{eq:intopt}
  q(R_{p,\delta_0,k}) = 2(\lambda-1)-\frac{(k-1-t)^2}{2(k-1)}.
\end{equation}
Choosing suitable $\lambda$ and $t$ and varying  $m \in \ZZ^+$ in \eqref{eq:propp} such that $\lambda \leq m$, one has that all rational numbers as on the right hand side of \eqref{eq:selfint}  are attained by Beauville-Bogomolov self--intersections  of optimal curves:

\begin{prop} \label{prop:tutteleint}
  Fix any integer $k\geq 2$. For any pair of integers $(\rho,\beta)$ with
$\rho \geq 0$ and $0 \leq \beta \leq k-1$, there are infinitely many positive integers $p$ such that $\Hilb^k(S)$ for general $(S,H) \in \K_p$ contains an optimal  curve $R_{p,\delta_0,k}$ 
with $q(R_{p,\delta_0,k})=2(\rho-1) - \frac{\beta^2}{2(k-1)}$. 
\end{prop}

This proposition  suggests a conceptual explanation for negative self-intersection numbers of certain extremal rays  (cf. \cite[beg. of \S\;4]{HTint}). 
If $k=2$, Proposition \ref {prop:tutteleint} gives the negative self-intersection numbers $-1/2$, $-2$ and $-5/2$ for optimal curves. This is in accordance with \cite[Conj. 3.1]{HT}, where Hassett and Tschinkel conjecture that the Mori cone should be generated by classes of curves $R$ with positive intersection with some polarizing class and such that either $q(R) \geq 0$ or $q(R)=-1/2$, $-2$ or $-5/2$. (It was proved in \cite{HT2} that the cone generated by these classes {\it contains} the Mori cone.)  
If $k=3$ we obtain the negative self-intersection numbers 
$-3$, $-9/4$, $-2$, $-1$ and $-1/4$, which are the numbers in \cite[Table H3]{HTint}, except for $-1$. In the case $k=4$ we obtain 
$-7/2$, $-8/3$, $-13/6$, $-2$, $-3/2$, $-2/3$, $-1/6$, which are the numbers in \cite[Table H4]{HTint}, except for $-3/2$.
Thus, our results provide extensions of the examples in 
\cite{HTint}, both to all $k \geq 2$ and to infinitely many  genera $p$.

\subsection{Curves of self-intersection zero and conjectures of 
Huybrechts and Sawon}\label{ssec:HS} The existence of curves (or divisors, cf.~\eqref{eq:qalfa}) with Beauville-Bogomolov self-intersection zero on a hyperk{\"a}hler manifold $X$ is conjectured  by Huybrechts \cite[\S\;21.4]{GHJ} and Sawon \cite[Conj.~4.2]{saw1} (see also \cite[Conj.~1]{saw2}) to imply that $X$ is birational to a Lagrangian fibration (Hassett and Tschinkel make in  \cite{HT} the same conjecture  in dimension $4$). The existence of a nontrivial {\it nef} divisor $D$ with $q(D)=0$ is a {\it necessary} condition for a hyperk{\"a}hler manifold to be a Lagrangian fibration (see e.g. \cite[\S~4.1 and Rmk.~4.2]{saw1}), and it is yet another conjecture of Sawon's \cite[Conj.~4.1]{saw1} that this is also a  {\it sufficient} condition. Both of these conjectures have been proved  for punctual Hilbert schemes of $K3$ surfaces by Bayer and Macr{\`i} in \cite[Thm.~1.5]{BM-MMP} after the submission of this paper.

In the case $X=\Hilb^k(S)$ with $(S,H) \in \K_p$ such that $\Pic (S) \cong \ZZ[H]$,  a divisor $mH-n\mathfrak{e}_k$ is isotropic if and only if $m^2(p-1)=n^2(k-1)$ and this occurs if and only if $(p-1)(k-1)$ is a square (cf. \cite[\S\;1]{saw2}). In the \emph{primitive}  case, i.e., $m=1$ and 
$p=n^2(k-1)+1$,  Sawon \cite[Thm.~2]{saw2} and Markushevich \cite[Cor.~4.4]{Ma} independently proved that $\Hilb^k(S)$ is  a Lagrangian fibration.  As we will see in Corollary \ref{cor:neclagfib} below, this result cannot be generalized as it stands to  $m>1$. 

A consequence of  Sawon-Markusevich's result is that Question \ref {cong2} has an affirmative answer in some more cases:

\begin{prop} \label{prop:intzero2}
  Let $(S,H)\in \mathcal K_p$ be general with $\Pic(S) \cong \ZZ[H]$ and $k \geq 2$ be any integer. Assume  that $p=n^2(k-1)+1$ for some $n \geq 2$. 
Then the extremal rays of 
the Mori cone of $\Hilb^k(S)$ are generated by $\mathfrak{r}_k$ and by the optimal 
class $R=R_{p,(n-1)^2(k-1)+1,k} \equiv H-2n(k-1)\mathfrak{r}_k$.
\end{prop}

\begin{proof}
Let $D$ be any isotropic divisor. Then  the class of $D$ is proportional to a multiple of $H-n\mathfrak{e}_k$ and, 
by \cite[Thm.~2]{saw2}, it is also  proportional to the class of the fiber of the Lagrangian fibration of $\Hilb^k(S)$, hence it is nef. One computes that  the minimal integer satisfying \eqref{eq:boundintro} is 
$\delta_0=(n-1)^2(k-1)+1$,  so that $R$ is effective, and $R \cdot D=0$, so that $R$ lies on the boundary of the Mori cone.
\end{proof}

\begin{remark} \label{rem:HT}
  This result also shows that the Mori cone is in some cases smaller than predicted by Hassett and Tschinkel in \cite[Conj. 1.2]{HTint}. Indeed, mimicking the idea in \cite[Rem.~9.4]{BM}, the class $\overline{R}:=R-\mathfrak{r}_k$
satisfies $q(\overline{R}) =-2n-\frac{1}{2(k-1)} \geq -\frac{k+3}{2}$ if $4n \leq k+2$. Since  
$(H-\epsilon\mathfrak{e}_k)\cdot\overline{R} >0$ and $H-\epsilon\mathfrak{e}_k$ is ample for small $\epsilon >0$, the class $\overline{R}$ should be effective by \cite[Conj. 1.2]{HTint}, which is not the case.     
\end{remark}

Curves  $R_{p,\delta,k}$ with self-intersection zero always exist
 when an isotropic divisor exists:

\begin{proposition} \label{prop:intzero} 
 Fix any integer $k\geq 2$.   Let $(S,H)\in \mathcal K_p$ be general with $\Pic(S) \cong \ZZ[H]$ and $p \geq 2$. 
Then there is a $\delta$ such that $q(R_{p,\delta,k}) = 0$ if and only if $(k-1)(p-1)=s^2$, for a positive integer $s$.  In this case $\delta=p-2s+k-1$.
\end{proposition}

\begin{proof}  By \eqref  {eq:selfint}, one has $q(R_{p,\delta,k}) = 0$ if and only if $(k-1)(p-1)=s^2$ and $\delta=p-2s+k-1$. Conversely, 
such $p$ and $\delta$ verify \eqref {eq:boundintro}.  
\end{proof}

\begin{cor} \label{cor:intzero}
  Let $(S,H)\in \mathcal K_p$ be general  with $\Pic(S) \cong \ZZ[H]$, $p \geq 2$ and $k\geq 2$. Assume that $D$ is a nontrivial isotropic divisor in $\Hilb^k(S)$
(so that $(k-1)(p-1)=s^2$, with $s \geq 1$). Then the rational curve $R=R_{p,p-2s+k-1,k}$ satisfies $D \cdot R=0$.
In particular, if 
\begin{equation}
  \label{eq:nonlagfib}
  (k-1)(\alpha+1)^2-(2s+1)(\alpha+1)+p \geq 0, \; \; \mbox{where} \; \; 
\alpha:=\Big \lfloor \frac{2s-k+1}{2(k-1)} \Big \rfloor, 
\end{equation}
then $D$ is not nef, since the rational curve $R'=R_{p,p-2s+k-2,k}$ satisfies 
$D \cdot R'<0$. 
\end{cor}

\begin{proof}
  Condition \eqref{eq:nonlagfib} is equivalent to $\rho \geq \alpha+1$ and the result follows by Remark \ref{rem:optimal}. 
\end{proof}

As a consequence, we obtain an additional necessary condition for $\Hilb^k(S)$ to be a Lagrangian fibration, showing that  Sawon-Markusevich's  mentioned result
cannot be generalized to all cases where an isotropic divisor exists:

\begin{cor} \label{cor:neclagfib}
  Let $(S,H)\in \mathcal K_p$, with  $p \geq 2$, and $k\geq 2$ be any integer. 
If  $\Hilb^k(S)$ is a Lagrangian fibration, then $(k-1)(p-1)=s^2$, with $s \geq 1$, and 
\begin{equation}
  \label{eq:lagfib}
  (k-1)(\alpha+1)^2-(2s+1)(\alpha+1)+p < 0, \; \; \mbox{where} \; \; 
\alpha:=\Big \lfloor \frac{2s-k+1}{2(k-1)} \Big \rfloor . 
\end{equation}
\end{cor}

\begin{question} \label{cong3}
Let $(S,H)\in \mathcal K_p$ be general. Is the condition \eqref{eq:lagfib}
also a {\it sufficient} condition for $\Hilb^k(S)$ to be a Lagrangian fibration?
\end{question}

An affirmative answer would mean that the Lagrangian fibration contracts the extremal ray $R_{p,p-2s+k-1,k}$. By the latest results of Bayer and Macr{\`i}  \cite[Thm.~1.5]{BM-MMP} after the submission of this paper, the question is reduced to whether the isotropic class is nef, which is in principle reduced to solving some Pell's equations, by \cite[\S\;13]{BM-MMP} .

%
%

\end{document}